\documentclass[11pt, oneside]{amsart}

\usepackage{tabularx, hyperref}
\usepackage{amssymb} \usepackage{amsfonts} \usepackage{amsmath}
\usepackage{amsthm} \usepackage{epsfig, subfig}
\usepackage{ amscd, amsxtra, latexsym}
\usepackage[all]{xy}
\usepackage{caption}
\usepackage{enumerate}
\usepackage{color}

\newtheorem{lemma}{Lemma}[section]
\newtheorem{thm}[lemma]{Theorem}
\newtheorem{prop}[lemma]{Proposition}
\newtheorem{cor}[lemma]{Corollary} 

\theoremstyle{definition}
\newtheorem{defn}[lemma]{Definition}
\newtheorem{rem}[lemma]{Remark}
\newtheorem{problem}[lemma]{Problem}

\def\@oddhead{\hfill \shorttitle \hfill \thepage}
\def\@evenhead{\thepage \hfill \shortauthor \hfill}
\def\@oddfoot{}
\def\@evenfoot{}

\newcommand{\G}{\Gamma}
\begin{document}

\title[QI-rigidity of piecewise geometric manifolds]{Quasi-isometric rigidity \\ of piecewise geometric manifolds}

\author[]{R. Frigerio}
\address{Dipartimento di Matematica, Universit\`a di Pisa, Largo B. Pontecorvo 5, 56127 Pisa, Italy}
\email{frigerio@dm.unipi.it}

\keywords{Prime decomposition, JSJ decomposition, relatively hyperbolic group, asymptotic cone, Bass-Serre tree, non-positive curvature, hyperbolic manifolds.}

\subjclass[2000]{20F65, 20F67, 20F69, 57M07, 57M50, 57M60, 20E06, 20E08, 53C23.}

\maketitle

\begin{abstract}
Two groups are virtually isomorphic if they can be obtained one from the other via a finite number of steps, where each step consists in taking a finite extension
or a finite index subgroup (or viceversa). Virtually isomorphic groups are always quasi-isometric, and a group $\G$ is quasi-isometrically rigid if every group quasi-isometric to $\G$ is virtually isomorphic to $\G$.
In this survey we describe quasi-isometric rigidity results for fundamental groups of manifolds which can be decomposed into geometric pieces. After stating by now classical results on  lattices in semisimple Lie groups, 
we focus on the class of fundamental groups of $3$-manifolds, and describe the behaviour of quasi-isometries with respect to the Milnor-Kneser prime decomposition (following Papasoglu and Whyte)
and with respect to the JSJ decomposition (following Kapovich and Leeb). We also discuss quasi-isometric rigidity results for fundamental groups
of higher dimensional graph manifolds, that were recently defined by Lafont, Sisto and the author. Our main tools are the study of geometric group actions and quasi-actions on Riemannian manifolds and on trees of spaces, via
the analysis of the induced actions on asymptotic cones.
\end{abstract}

\section{Introduction}
The study of finitely generated groups up to quasi-isometry is a broad and very active research area which links group theory to geometric topology. As stated, the task of 
classifying groups up to quasi-isometry is certainly too ambitious. Nevertheless, when restricting to classes of groups coming from specific algebraic or geometric contexts,
several astonishing results have been obtained in the last three decades. In this survey we will go through some of them, paying a particular attention to the case of fundamental groups
of manifolds obtained by gluing ``geometric'' pieces.

Once a symmetric finite set $S$ of generators for a group $\Gamma$ has been fixed, one can construct a graph having as vertices the elements of $\G$, in such a way that two vertices
are connected by an edge whenever the corresponding elements of the group are obtained one from the other by right multiplication by a generator. The resulting
graph $C_S(\G)$ is the \emph{Cayley graph of $\G$ (with respect to $S$)}. Its small-scale structure depends on the choice of the set of generators $S$, but the quasi-isometry class
of $C_S(\G)$ is actually independent of $S$, so  the quasi-isometry class of the group $\G$ is well-defined (see Section~\ref{definitions:sec} for the definition of quasi-isometry and for further details). 
It is natural to ask which algebraic and/or geometric properties of the group $\G$ may be encoded by the quasi-isometry class of its Cayley graphs. Another way to phrase the same question
is the following: if $\G,\G'$ are quasi-isometric groups, what sort of properties must $\G$ share with $\G'$? Surprisingly enough, it turns out that the (apparently quite loose) relation of being quasi-isometric
implies many other relations which may look much finer at first glance (see the paragraph after Definition~\ref{rigidity:def} and Subsection~\ref{lattices} 
for a list of many results in this spirit).

In this paper, we will mainly focus on the particular case of fundamental groups
of manifolds which decompose into specific geometric pieces. For the sake of simplicity, henceforth we confine ourselves to the case of \emph{orientable} manifolds. 
In dimension 3 the picture is particularly well understood. 
The now proved  Thurston's geometrization conjecture implies that every closed (i.e.~compact without boundary) $3$-manifold $M$ can be canonically decomposed into locally homogeneous Riemannian manifolds.
Namely, Milnor--Kneser prime decomposition Theorem~\cite{Milnor} implies that 
$M$ can first be cut along spheres into summands $M_1,\ldots,M_k$ such that each $M_i$ (after filling the resulting boundary spheres with disks) either is homeomorphic to $S^2\times S^1$ or is \emph{irreducible}, meaning that every $2$-sphere in it bounds a $3$-ball.
Then, every irreducible manifold admits a canonical decomposition along a finite family of disjoint tori, named \emph{JSJ decomposition} after Jaco--Shalen and Johansson~\cite{JS1,JS2,JS3,Joha}, such that every component obtained
by cutting along the tori is either Seifert fibered or hyperbolic.
%(or it admits a geometric structure modeled on $Sol$ -- this case occurs only when the decomposition is trivial). 
A quite vague but natural question is then the following: can quasi-isometries detect the canonical decomposition
of $3$-manifolds? And what can one say in higher dimensions? The main aim of this paper is to survey the work done by many mathematicians in order to provide answers to these questions.

Before going on, let us first describe some phenomenon that quasi-isometries \emph{cannot} detect. 
It readily follows from the definition that every finite group is quasi-isometric to the trivial group, but even more is true.
Following~\cite{DK}, we say that two groups $\G_1,\G_2$ are \emph{virtually isomorphic} if there exist finite index subgroups
$H_i<\G_i$ and finite normal subgroups $F_i\triangleleft H_i$, $i=1,2$, such that $H_1/F_1$ is isomorphic to $H_2/F_2$.
It is an exercise to check that virtual isomorphism is indeed an equivalence relation. In fact, it is the smallest equivalence relation for which
any group is equivalent to any of its finite index subgroups, and any group is equivalent to any of its finite extensions, where
we recall that $\G$ is a finite extension of  $\Gamma'$ if it fits into an exact sequence of the form
$$
1 \longrightarrow F\longrightarrow \G\longrightarrow \G'\longrightarrow 1\ ,
$$
where $F$ is finite. It is not difficult to show that virtually isomorphic groups are quasi-isometric (see Remark~\ref{vi-qi}). On the other hand, being virtually isomorphic
is a very strict condition, so there is no reason why  quasi-isometric groups
should be virtually isomorphic in general. The following definition singles out those situations where quasi-isometry implies virtual isomorphism.

\begin{defn}\label{rigidity:def}
 A group $\G$ is \emph{quasi-isometrically rigid} (or QI-rigid for short) if any group quasi-isometric to $\G$ is in fact virtually isomorphic to $\G$. A collection of groups
 $\mathcal{C}$ is \emph{quasi-isometrically rigid} (or QI-rigid for short) if any group quasi-isometric to a group in $\mathcal{C}$ is virtually isomorphic to (possibly another) group in $\mathcal{C}$.
\end{defn}

We are not giving here the most complete list of groups or of classes of groups that are (or that are not) quasi-isometrically rigid (the interested reader is addressed e.g.~to~\cite{Kap:qi}
or to~\cite[Chapter 23]{DK}). However, it is maybe worth mentioning at least some results that can help the reader to put the subject of this survey in a more general context. 
Some important properties of groups are immediately checked to be preserved by quasi-isometries (this is the case, for example, for amenability). Moreover,
many algebraic properties are preserved by quasi-isometries.
For example, free groups are QI-rigid~\cite{Sta1, Dun} (and this fact is closely related to our discussion of the invariance of the prime decomposition under quasi-isometries in Appendix~\ref{prime:sec}), 
as well as nilpotent groups~\cite{Gro-nilpotent} and abelian groups~\cite{Pansu}. On the contrary, the class of solvable groups is \emph{not} QI-rigid~\cite{Dyu}.
Also algorithmic properties of groups are often visible to quasi-isometries: for example,
the class of finitely presented groups with solvable word problem is QI-rigid (in fact, the class of finitely presented groups is QI-rigid, and the growth type of the Dehn function of a finitely presented group
is a quasi-isometry invariant -- see e.g.~the nice survey by Bridson~\cite{Bri-word}).

\subsection{QI-rigidity of lattices in semisimple Lie groups}\label{lattices}

In this survey we are mainly interested in quasi-isometric rigidity of fundamental groups of manifolds with specific
geometric properties. At least in the case of locally symmetric spaces (which are in some sense the most regular manifolds, hence the most natural spaces to first deal with)
the subject is now completely understood, thanks to the contribution of several mathematicians. The general strategy that lead to the complete classification up to quasi-isometry of lattices in semisimple
Lie groups applies also to the case we are interested in, so we briefly discuss it here. First of all, the fundamental Milnor-Svarc Lemma (see Proposition~\ref{milsv}) asserts
that the fundamental group of any closed Riemannian manifold is quasi-isometric to the universal covering of the manifold. Therefore, uniform lattices in the same semisimple Lie group $G$ are 
quasi-isometric, since they are all quasi-isometric to the symmetric space $X$ associated to $G$. Moreover, the classification of groups quasi-isometric to such lattices coincides with the classification
of groups quasi-isometric to $X$. Now, whenever a group $\Gamma$ is quasi-isometric to a geodesic space $X$, it is 
possible to construct a geometric \emph{quasi-action} of $\Gamma$ on $X$ by quasi-isometries (see Proposition~\ref{quasiact:prop}). In order to show quasi-isometric rigidity of uniform lattices in $G$, one would like to 
turn this quasi-action into a genuine action by isometries. 
Having shown this,
one is then provided with a homomorphism of $\Gamma$ into $G$, and it is routine to show that such a homomorphism
has a finite kernel and a discrete image. 
In the case when $X$ has higher rank, QI-rigidity then holds because every quasi-isometry stays at bounded distance from an isometry~\cite{klelee,EF}. The same is true also for every quaternionic
hyperbolic space of dimension at least two~\cite{Pansu2}, while in the case of real or complex hyperbolic spaces there exist plenty of quasi-isometries which are not at finite distance from any isometry. 
In the real hyperbolic case, as already Gromov pointed out~\cite{Gromov-infinite}, Sullivan's and Tukia's results on \emph{uniformly} quasi-conformal groups~\cite{Tukia,Sullivan} imply that, in dimension strictly bigger than $2$, 
every group quasi-action by \emph{uniform} quasi-isometries is just the perturbation of a genuine isometric action. 
This concludes the proof of QI-rigidity for the class of uniform lattices in $G={\rm Isom}(\mathbb{H}^n)$, $n\geq 3$ (see also~\cite{CC}). 
The complex hyperbolic case is settled thanks to an analogous argument due to Chow~\cite{Chow}.
Finally, QI-rigidity also holds for surface groups
(i.e.~for uniform lattices in $G={\rm Isom}(\mathbb{H}^2)$) thanks to a substantially different proof obtained from the combination of results
by Tukia~\cite{Tu2}, Gabai~\cite{Gabai} and Casson--Jungreis~\cite{CJ}.
Summarizing, we have the following:

\begin{thm}
 Let $G$ be a semisimple Lie group. Then the class of irreducible uniform lattices in $G$ is QI-rigid.
\end{thm}

Note however that, since every uniform lattice in $G$ is quasi-isometric to the associated symmetric space $X$, a single uniform lattice
is usually very far from being QI-rigid. Surprisingly enough, much more can be said when dealing with \emph{non-uniform} lattices. In fact, a non-uniform lattice $\Gamma$ in $G$ is quasi-isometric
to the complement in the corresponding symmetric space $X$  of an equivariant collection of horoballs. Such a space is much more rigid than $X$ itself, and this fact can be exploited to show the following:

\begin{thm}[\cite{Schwartz,Schwartz2,EF,Eskin}]
 Let $G$ be a semisimple Lie group distinct from $SL(2,\mathbb{R})$. Then every non-uniform lattice in $G$ is QI-rigid. 
\end{thm}

The fact that patterns inside $X$ (such as the previously mentioned collection of horoballs) can be of use in proving rigidity of maps
will clearly emerge in the  discussion of quasi-isometric rigidity of fundamental groups of manifolds with distinguished submanifolds
(such as the tori of the JSJ-decomposition of irreducible $3$-manifolds, or the tori separating the pieces of higher dimensional graph manifolds introduced in~\cite{FLS}).

For a detailed survey on QI-rigidity of lattices in semisimple Lie groups we refer the reader to~\cite{Farb}.

\subsection{Quasi-isometric invariance of the prime decomposition}
Let us recall that every manifold will be assumed to be orientable throughout the paper. 
A $3$-manifold $M$ is \emph{prime} if it does not admit any non-trivial decomposition as a connected sum (a connected sum is trivial if one of its summands is the $3$-sphere $S^3$),
and it is irreducible if every $2$-sphere embedded in $M$ bounds an embedded $3$-ball. Every irreducible manifold is prime, and the only prime manifold which is not irreducible is $S^2\times S^1$. 
As mentioned above, every closed $3$-manifold $M$
admits a canonical decomposition $M=M_1\#\ldots\# M_k$ into 
a finite number of prime summands. Of course we have $\pi_1(M)=\pi_1(M_1)*\ldots *\pi_1(M_k)$, and on the other hand the (proved) Kneser conjecture  asserts
that every splitting of the fundamental group of a closed $3$-manifold as a free product is induced by a splitting of the manifold as a connected sum (see e.g.~\cite[Theorem 7.1]{Hempel}).
As a consequence, the study of the quasi-isometry invariance of the prime decomposition boils down to the
study of the behaviour of quasi-isometries with respect to 
free decompositions of groups, a task that has been completely addressed by Papasoglu and Whyte in~\cite{PapaWhyte}.

It is not  true that the quasi-isometry type of the fundamental group recongnizes whether a manifold is prime, since $\pi_1(S^2\times S^1)=\mathbb{Z}$ and 
$\pi_1(\mathbb{P}^3(\mathbb{R})\#\mathbb{P}^3(\mathbb{R}))=\mathbb{Z}\rtimes \mathbb{Z}_2$ are obviously virtually isomorphic (however, Theorem~\ref{prime:thm} implies that irreducibility is detected by
the quasi-isometry class of the fundamental group). Moreover, if $\G_1,\G_2$ contain at least three elements and
$F$ is virtually cyclic, then $\G_1 *\G_2 * F$ is quasi-isometric to $\G_1 *\G_2$~\cite{PapaWhyte}, so quasi-isometries cannot see in general all the pieces of the prime decomposition of a manifold. However, the following result shows that 
quasi-invariance
of the prime decomposition with respect to quasi-isometries holds as strongly as the just mentioned  examples allow. For the sake of convenience, let us say that
a prime manifold is \emph{big} if its fundamental group is not virtually cyclic (in particular, it is infinite). 

\begin{thm}\label{prime:thm}
 Let $M,M'$ be closed orientable $3$-manifolds. Then $\pi_1(M)$ is quasi-isometric to $\pi_1(M')$ if and only if one of the following mutually exclusive
conditions holds:
\begin{enumerate}
 \item Both $M$ and $M'$ are prime with finite fundamental group.
 \item Both $M$ and $M'$ are irreducible with infinite quasi-isometric fundamental groups. 
 \item $M$ and $M'$ belong to the set $\{S^2\times S^1,\mathbb{P}^3(\mathbb{R})\#\mathbb{P}^3(\mathbb{R})\}$.
 \item both $M$ and $M'$ are not prime and distinct from $\mathbb{P}^3(\mathbb{R})\#\mathbb{P}^3(\mathbb{R})$; moreover, if $N$ is a big piece in the prime decomposition of $M$, then there exists a big piece
 $N'$ in the prime decomposition of $M'$ such that $\pi_1(N)$ is quasi-isometric to $\pi_1(N')$, and viceversa.
\end{enumerate} 
 \end{thm}

 In the spirit of Definition~\ref{rigidity:def}, one may also ask what can be said about groups which are quasi-isometric to the fundamental group of a closed $3$-manifold. The following holds:

 \begin{thm}\label{prime2:thm}
  Let $M$ be a closed non-prime $3$-manifold distinct from $\mathbb{P}^3(\mathbb{R})\#\mathbb{P}^3(\mathbb{R})$, and let $\G$ be a finitely generated group. Then $\G$ is quasi-isometric to $\pi_1(M)$ if and only if 
  $\G$ has infinitely many ends  and 
  splits as the fundamental group of a graph of groups with finite edge groups such that
  the set of quasi-isometry types of one-ended vertex groups coincides with the set of quasi-isometry types of big summands in the prime decomposition of $M$. 
 \end{thm}

Theorems~\ref{prime:thm} and~\ref{prime2:thm} readily descend from~\cite{PapaWhyte}.  
 Nevertheless, since apparently they are not available in an explicit form elsewhere, we will provide a quick proof of them in the appendix.
By Theorems~\ref{prime:thm} and~\ref{prime2:thm}, the problem of classifying fundamental groups of $3$-manifolds up to quasi-isometry is now reduced to the study of irreducible ones. 
 
 \subsection{QI-rigidity of the eight $3$-dimensional geometries}\label{8geom:subsec}
As already mentioned, the now proved Thurston's geometrization conjecture states that every closed irreducible $3$-manifold decomposes along a finite family of disjoint $\pi_1$-injective tori (and Klein bottles)
into a finite collection of manifolds each of which admits a complete finite-volume locally homogeneous Riemannian metric (henceforth, such a manifold will be said to be \emph{geometric}). 
Strictly speaking, this decomposition (which will be called \emph{geometric} form now on) is not exactly the same as the classical JSJ decomposition of the manifold, 
the only differences being the following ones: 
\begin{enumerate}
\item
If $M$ is a closed manifold locally isometric to $Sol$ (see below), then 
the geometric decomposition of $M$ is trivial. On the other hand,
$M$ either is a torus bundle over $S^1$, or it is obtained by gluing along their toric boundaries two 
twisted orientable $I$-bundles over the Klein bottle. In the first case, the JSJ decomposition of $M$ is obtained by cutting along a fiber of the bundle, in the second
one it is obtained by cutting along the torus that lies bewteen the $I$-bundles. 
\item
If $M$ is not as in (1), then
each torus of the JSJ decomposition which bounds
a twisted orientable $I$-bundle over a Klein bottle is replaced in the geometric decomposition
by the core Klein bottle of the bundle.
\end{enumerate}
In fact, neither (the internal part of) the twisted orientable $I$-bundle over the Klein bottle nor $S^1\times S^1\times (0,1)$ supports a locally homogeneous Riemannian
metric of finite volume.
The geometric decomposition is still canonical (i.e.~preserved by homeomorphisms) and cuts the manifolds into geometric pieces in the sense described above.

Up to equivalence, there exist exactly eight homogeneous Riemannian $3$-manifolds which admit a finite-volume quotient:
the constant curvature spaces $S^3,\mathbb{R}^3,\mathbb{H}^3$, the product spaces $S^2\times \mathbb{R}$, $\mathbb{H}^2\times \mathbb{R}$, and the $3$-dimensional Lie groups $\widetilde{SL_2}=
\widetilde{SL_2(\mathbb{R})}$, $Sol$, $Nil$~\cite{Thurston, Scott} (the equivalence that needs to be taken into account is not quite the relation of being homothetic, because 
a single Lie group may admit several non-homothetic left-invariant
metrics, and we don't want to consider the same Lie group endowed with two non-homotetic metrics as two distinct geometries).
We will call these spaces ``geometries'' or ``models''. We will say that a manifold is \emph{geometric} if it admits a complete finite-volume Riemannian metric locally isometric to one of the eight models just introduced
(with the exception of $\mathbb{P}^3(\mathbb{R})\# \mathbb{P}^3(\mathbb{R})$ and of $S^2\times S^1$, every geometric manifold is automatically irreducible).
It is well-known that, with the exception of $\mathbb{H}^2\times \mathbb{R}$ that is quasi-isometric to $\widetilde{SL_2}$, the eight $3$-dimensional geometries are pairwise not quasi-isometric (see Proposition~\ref{8models}). 
Together with Milnor--Svarc Lemma, this information is sufficient to classify irreducible $3$-manifolds with trivial geometric decomposition, up to quasi-isometries of their fundamental groups:

\begin{prop}\label{8geom:prop}
 Let $M_1,M_2$ be closed irreducible  $3$-manifolds, and suppose that $M_i$ admits a locally homogeneous Riemannian metric locally isometric to $X_i$, where $X_i$ is one of the eight
 geometries described above. Then $\pi_1(M_1)$ is quasi-isometric to $\pi_1(M_2)$ if and only if either $X_1=X_2$ or $X_1,X_2\in\{\widetilde{SL_2},\mathbb{H}^2\times\mathbb{R}\}$. 
\end{prop}

Together with the fact that a closed $3$-manifold cannot support both a structure locally modeled on $\mathbb{H}^2\times\mathbb{R}$ and
a structure modeled on $\widetilde{SL_2}$, Proposition~\ref{8geom:prop} implies tha the geometry of a single closed irreducible manifold with trivial geometric decomposition is uniquely determined by its fundamental group.

We have already seen  that the class of fundamental groups of closed hyperbolic $3$-manifolds is QI-rigid.
The following result extends quasi-isometric rigidity to   
the remaining seven 
$3$-dimensional geometries.

 \begin{thm}\label{8geom:rigidity}
 Let $X$ be any of the eight $3$-dimensional geometries, and let $\Gamma$ be a group quasi-isometric to the fundamental group of a closed manifold modeled on $X$.
 If $X\neq \mathbb{H}^2\times \mathbb{R}$ and $X\neq\widetilde{SL_2(\mathbb{R})}$, then $\Gamma$ is virtually isomorphic to the fundamental group of a closed manifold modeled on $X$.
 If $X=\mathbb{H}^2\times \mathbb{R}$ or $X=\widetilde{SL_2(\mathbb{R})}$, then
 $\Gamma$ is virtually isomorphic to the fundamental group of a closed manifold modeled on $\mathbb{H}^2\times \mathbb{R}$ or on $\widetilde{SL_2(\mathbb{R})}$.
 \end{thm}
 \begin{proof}
 The cases
of $S^3$ and of $S^2\times\mathbb{R}$ are obvious due to the quasi-isometric rigidity of cyclic groups.
   The case of $\mathbb{H}^2\times\mathbb{R}$ and $\widetilde{SL_2(\mathbb{R})}$ is due to Rieffel~\cite{Rieffel} (see also~\cite{klelee2}), 
   while quasi-isometric rigidity of abelian groups follows from~\cite{Gro-nilpotent,Pansu}.
   The case of $Sol$ is settled in~\cite{EFW}. Finally, quasi-isometric rigidity of $Nil$ may be proved as follows.
   Let $\G$ be quasi-isometric to $Nil$. Quasi-isometric rigidity of nilpotent groups~\cite{Gro-nilpotent} implies that $\G$ is virtually nilpotent. Since every finitely generated nilpotent group is linear,
   by Selberg Lemma we may suppose, up to virtual isomorphism, that $\G$ is nilpotent and torsion-free. 
   The Malcev closure $G$ of $\G$ is a simply connected nilpotent Lie group in which $\G$ embeds as a uniform lattice~\cite{Rag}. 
  Being quasi-isometric, the groups $\G$, $G$  and $Nil$ must have equivalent growth functions. Now results of Guivarc'h~\cite{Guiv} imply that 
  the growth of $Nil$ is polynomial of degree $4$, and that the only simply connected nilpotent Lie groups with polynomial growth of degree $4$ are $Nil$ and $\mathbb{R}^4$.
  In order to conclude it is now sufficient to observe that the case $G=\mathbb{R}^4$ cannot hold, since otherwise the Heisenberg group (which is a uniform lattice in $Nil$) would be quasi-isometric
  to $\mathbb{Z}^4$, whence virtually abelian.
 \end{proof}

\subsection{Quasi-isometric invariance of the JSJ decomposition}\label{JSJ:subsec}

Let $M$ be a closed irreducible $3$-manifold.
We are now interested in the case when the decomposition of $M$ is not trivial. In this case we say that the manifold is \emph{non-geometric} (because indeed it does not support
any complete finite-volume locally homogeneous metric),
and all the pieces resulting from cutting along the tori and the Klein bottles
of the geometric decomposition are the internal parts of compact irreducible $3$-manifolds bounded by tori. As it is customary in the literature, 
we will often confuse the concrete pieces of the decomposition (which are open) with their natural compactifications obtained by adding some boundary tori.
The possibilities for the geometry of these pieces are much more restricted than in the closed case: the only geometries that admit non-compact  finite-volume quotients (whose compactification is) bounded by tori
are $\mathbb{H}^3$, $\widetilde{SL_2}$ and $\mathbb{H}^2\times \mathbb{R}$. Moreover, if the internal part of a compact irreducible $3$-manifold bounded by tori can be modeled on $\widetilde{SL_2}$, then it can be modeled
also on $\mathbb{H}^2\times\mathbb{R}$, and viceversa (while as in the closed case, complete finite-volume hyperbolic manifolds cannot support a metric locally isometric to $\mathbb{H}^2\times\mathbb{R}$).
Pieces modeled on $\mathbb{H}^2\times \mathbb{R}$ admit a foliation by circles, and are called
\emph{Seifert fibered} (or simply \emph{Seifert}).
They are finitely covered (as foliated manifolds) by manifolds of the form $\Sigma\times S^1$ (endowed with the obvious 
foliation by circles), where $\Sigma$ is a punctured
surface of finite type with negative Euler characteristic.
The first question that comes to mind is then whether
quasi-isometries of the fundamental group can detect the presence or the absence of hyperbolic pieces and/or of Seifert fibered pieces. 
This question was answered in the positive by Kapovich and Leeb in~\cite{kapleenew}. Before stating their result, let us 
recall that, if $N\subseteq M$ is a piece of the geometric decomposition of $M$, then the inclusion $N\hookrightarrow M$ induces
an injective map on fundamental groups. Therefore, $\pi_1(N)$ may be identified with (the conjugacy class of) a subgroup of $\pi_1(M)$, and this fact will be always tacitly understood in the sequel.
Kapovich and Leeb first proved that 
quasi-isometries recognize the presence of Seifert fibered pieces~\cite{kaplee0}, and then they improved their result as follows:

\begin{thm}[\cite{kapleenew}]\label{kaplee}
 Let $M,M'$ be closed irreducible $3$-manifolds, let $f\colon \pi_1(M)\to \pi_1(M')$ be a quasi-isometry, and let $N\subseteq M$ be a piece of the geometric decomposition of $M$. Then there exists a piece
 $N'$ of the geometric decomposition of $M'$ such that $f(\pi_1(N))$ stays at bounded Hausdorff distance from (a conjugate of) $\pi_1(N')$. Moreover, $f|_{\pi_1(N)}$ is at finite distance from a quasi-isometry between $\pi_1(N)$
 and $\pi_1(N')$.
\end{thm}

 In particular, the set of the quasi-isometry types of the fundamental groups of the pieces of an irreducible manifold is a quasi-isometry invariant of the fundamental group of the manifold. 
 Since Seifert manifolds with non-empty boundary and cusped hyperbolic $3$-manifolds cannot have quasi-isometric fundamental groups
 (see Corollary~\ref{aaa}), it readily follows that the presence of a hyperbolic and/or
 a Seifert piece is a quasi-isometry invariant of the fundamental group (another proof of the fact that quasi-isometries detect the presence of a hyperbolic piece was given by
 Gersten~\cite{Gersten}, who showed that $M$ contains a hyperbolic piece if and only if the divergence of $\pi_1(M)$ is exponential).

 Regarding quasi-isometric rigidity, in the same paper Kapovich and Leeb proved the following:
 
 \begin{thm}[\cite{kapleenew}]\label{KL2}
  The class of fundamental groups of closed non-geometric irreducible $3$-manifolds is QI-rigid.
 \end{thm}

 Theorems~\ref{kaplee} and~\ref{KL2} will be the main object of this survey. Following Kapovich and Leeb, we will make an extensive use of asymptotic cones, a very useful quasi-isometry invariant
 of geodesic spaces introduced by Gromov and studied by a number of mathematicians in the last decades. Nowadays, asymptotic cones of groups themselves are the subject of 
 an independent research field. We will introduce them and briefly describe
 their properties in Section~\ref{asymptotic:sec}.

 The following result by Behrstock and Neumann shows that Theorem~\ref{kaplee} may be strengthened in the case of \emph{graph manifolds}, i.e.~of irreducible $3$-manifolds whose geometric decomposition is not trivial and does not contain any hyperbolic piece.
 
 \begin{thm}[\cite{BN1}]\label{BN1}
  Let $M,M'$ be closed non-geometric graph manifolds. Then $\pi_1(M)$ is quasi-isometric to $\pi_1(M')$. 
 \end{thm}
 
 Another big progress towards the classification of quasi-isometry types of non-geometric irreducible $3$-manifolds has been done in a subsequent paper by Behrstock and Neumann~\cite{BN2}, where the case when at least one of the hyperbolic
 pieces of the decomposition is \emph{not} arithmetic is completely addressed. Since their results cannot be properly stated without introducing a bit of terminology, and since
 our attention in this paper will be primarily concentrated on Theorems~\ref{kaplee} and~\ref{KL2} above, we address the reader to~\cite{BN2} for the precise statements. 

Before jumping into the realm of higher dimensional manifolds, let us stress the fact that all the theorems stated in this subsection make an essential use of the geometry supported by the pieces in which an irreducible manifold decomposes.
While in the case of the prime decomposition of $3$-manifolds the fact that quasi-isometries recognize summands can be deduced from more general facts regarding the quasi-isometry classification of free products,
no general theorem on amalgamated products is known which can be exploited to show that quasi-isometries also detect
the pieces of the geometric decomposition of irreducible $3$-manifolds. 
During the last 25 years the notion of JSJ decomposition has been extended from the case of (fundamental groups of) irreducible $3$-manifolds to the context of finitely presented groups. 
The JSJ decomposition was first defined by Sela for torsion-free word-hyperbolic groups~\cite{Sela}.
After Bowditch developed the theory for general word-hyperbolic groups \cite{Bow}, 
other notions of JSJ decomposition were introduced by 
Rips and Sela~\cite{RS}, Dunwoody and Sageev~\cite{DunSa}, Scott and Swarup~\cite{SS}, and 
  Fujiwara and Papasoglu~\cite{FP} in the general context of finitely presented groups.
However, uniqueness of such decompositions up to isomorphism is already  delicate (see e.g.~Forester \cite{Forester} and Guirardel and Levitt \cite{GL}), and no quasi-isometry invariance result may be directly applied
to the case we are interested in. Moreover,
in many versions of the theory the outcome of the decomposition is a graph of groups with two-ended (i.e.~infinite virtually cyclic) subgroups as edge groups, rather than rank-2 free abelian edge groups
like in the classical case. We refer the reader e.g.~to~\cite{Pap,Papas,MSW1,MSW2} for some approaches  to quasi-isometry invariance of decompositions of groups as amalgamated products.

\subsection{A glimpse to the higher-dimensional case}
In dimension greater than three generic manifolds do not admit decompositions into pieces with controlled geometry. Therefore, in order to study to what extent quasi-isometries capture the geometry of
the pieces of a manifold obtained by gluing ``geometric pieces'' one first need to define the appropriate class of objects to work with.
For example, Nguyen Phan defined in~\cite{Nguyen1,Nguyen2}
classes of manifolds obtained by gluing non-positively curved finite-volume locally symmetric spaces with cusps removed. In the cited papers, Nguyen Phan proved (smooth) rigidity results for manifolds obtained this way. 
In the context of non-positively curved manifolds, Leeb and Scott defined a canonical decomposition  along embedded flat manifolds which is meant to generalize to higher dimensions the JSJ decomposition of irreducible $3$-manifolds~\cite{LS}. 
A different class of manifolds, called \emph{higher dimensional graph manifolds}, was defined by Lafont, Sisto and the author in~\cite{FLS} as follows.

\begin{defn}
A compact smooth $n$-manifold $M$, $n\geq 3$, is a \emph{higher dimensional graph manifold} (or \emph{HDG manifold} for short) 
if it can be constructed in the following way:
\begin{enumerate}
\item
For every $i=1,\ldots, r$, take a complete finite-volume
non-compact hyperbolic $n_i$-manifold $N_i$ 
with toric cusps, where $3\leq n_i\leq n$.
\item
Denote by $\overline{N}_i$ the manifold obtained by ``truncating the cusps'' of $N_i$, i.e.~by removing
from $N_i$ a horospherical neighbourhood of each cusp.
\item
Take the product $V_i=\overline{N}_i\times T^{n-n_i}$, where $T^k=(S^1)^k$ is the $k$-dimensional torus. 
\item
Fix a pairing of some boundary components of the $V_i$'s and glue the paired
boundary components via diffeomorphisms, so as to obtain a connected manifold of
dimension $n$.
\end{enumerate}
Observe that $\partial M$ is either empty or consists of tori. The submanifolds $V_1,\ldots,V_r$ are
the \emph{pieces} of $M$. The manifold $\overline{N}_i$ is the \emph{base} of $V_i$, while every
subset of the form $\{\ast\}\times T^{n-n_i}\subseteq V_i$ is a \emph{fiber} of $V_i$. The boundary
tori which are identified together are the \emph{internal walls} of $M$ (so any two distinct pieces in $M$
are separated by a collection of internal walls), while the components of $\partial M$ are the \emph{boundary walls} of $M$.
\end{defn}

Therefore, HDG manifolds can be decomposed into pieces, each of which supports 
a finite-volume product metric locally modeled on some $\mathbb{H}^k \times \mathbb{R} ^{n-k}$ ($k\geq 3$).

\begin{rem}
In the original definition of HDG manifold the gluing diffeomorphisms between the paired tori are
required to be affine with respect to the canonical affine structure that is defined on the boundary components of the pieces. Without this restriction,
no smooth rigidity could hold for HDG manifolds. Nevertheless, the quasi-isometric rigidity results we are investigating here are not affected by allowing generic diffeomorphisms
as gluing maps.
\end{rem}

Probably the easiest (and somewhat less interesting) examples of HDG manifolds are the so-called \emph{purely hyperbolic} HDG manifolds, i.e.~those HDG manifolds whose pieces are all just truncated hyperbolic manifolds. 
Such manifolds enjoy additional nice properties, that are of great help in understanding
their geometry: for example, they support non-positively curved Riemannian metrics (at least when the gluings are affine),
and their fundamental groups are relatively hyperbolic~\cite{FLS}. Examples of purely hyperbolic HDG manifolds include both the classical ``double'' of a finite volume 
hyperbolic manifold with toric cusps, as well as twisted doubles of such manifolds (in the sense of 
Aravinda and Farrell \cite{ArFa}).

A number of rigidity results (like smooth rigidity \`a la Mostow within the class of HDG manifolds obtained via affine gluings and topological rigidity \`a la Farrell--Jones within the class of aspherical manifolds)
are proved in~\cite{FLS} for generic HDG manifolds. In order to get quasi-isometric rigidity, however, a further assumption is needed, which ensures that the (fundamental group of) each piece
is undistorted in the fundamental group of the manifold. 
Let $M$ be a HDG graph manifold, and $V^+, V^-$ a pair of adjacent (not necessarily distinct) pieces of $M$. 
We say that the two pieces have {\it transverse fibers} along the common internal wall $T$ provided that, under the gluing diffeomorphism 
$\psi: T^+ \rightarrow T^-$ of the paired boundary tori corresponding to $T$, the image of the fiber subgroup of $\pi_1(T^+)$ under $\psi_*$ intersects 
the fiber subgroup of $\pi_1(T^-)$ only in $\{0\}$.

\begin{defn}\label{irr:def}
A HDG manifold is {\it irreducible} if every pair of adjacent
pieces has transverse fibers along every common internal wall.
\end{defn}

In the case of $1$-dimensional fibers (and when restricting only to affine gluings between the pieces), a HDG graph manifold is irreducible
if and only if the $S^1$-bundle structure on each piece cannot be extended
to the union of adjacent pieces, a phenomemon which always occurs in $3$-dimensional graph manifolds built up by products of
hyperbolic surfaces times $S^1$. This suggests that irreducible HDG manifolds should provide a closer analogue to $3$-dimensional graph manifolds
than generic HDG manifolds.

Every purely hyperbolic HDG manifold
is  irreducible. However, the class of irreducible HDG manifolds is much richer than the class of purely hyperbolic ones:
for example, in each dimension $n\geq 4$, there exist infinitely many irreducible HDG $n$-manifolds
which do not support any locally CAT(0) metric, in contrast with the fact that every purely hyperbolic HDG manifold (obtained via affine gluings)
supports a non-positively curved Riemannian metric.

In this paper we will be mainly concerned with the behaviour of the fundamental groups of HDG manifolds with respect to quasi-isometries. Using the technology of asymptotic cones 
(together with other tools, of course), in Section~\ref{QIrigidity:sec} we will sketch the proof of the following results, which are taken from~\cite{FLS}.
The following result provides a higher-dimensional analogue of Theorem~\ref{kaplee}:

\begin{thm}\label{qi-preserve:thm}
Let $M_1$, $M_2$ be a pair of irreducible HDG manifolds, and $\G_i=\pi_1(M_i)$ their respective
fundamental groups.  Let $\Lambda_1 \leq \G_1$ be a subgroup conjugate to the fundamental
group of a piece in $M_1$, and $\varphi\colon: \G_1\rightarrow \G_2$ be a quasi-isometry. 
Then, the set $\varphi(\Lambda_1)$ is within finite Hausdorff distance from
a conjugate of $\Lambda_2 \leq \G_2$, where $\Lambda_2$ is the fundamental group of a piece  in $M_2$.
\end{thm}

After one knows that
fundamental groups of the pieces are essentially mapped to fundamental group of the pieces under quasi-isometries, the next goal is
to understand the behavior of groups quasi-isometric to the fundamental group of a piece. 

\begin{thm}\label{product:thm}
Let $N$ be a complete finite-volume hyperbolic $m$-manifold, $m\geq 3$, and
let $\Gamma$ be a finitely generated group quasi-isometric to
$\pi_1 (N)\times\mathbb{Z}^d$, $d\geq 0$. 
Then there exist a finite-index subgroup $\Gamma'$ of $\Gamma$,
a finite-sheeted covering $N'$ of $N$, a group $\Delta$  and a finite group $F$ 
such that the following short exact sequences hold:
\begin{equation*}
\xymatrix{
1\ar[r] &\mathbb{Z}^d \ar[r]^j & \Gamma' \ar[r] & \Delta \ar[r] & 1,\\
}
\end{equation*}
\begin{equation*}
\xymatrix{
1\ar[r] & F \ar[r] & \Delta\ar[r] & \pi_1 (N')\ar[r] & 1 .
}
\end{equation*}
Moreover, $j(\mathbb{Z}^d)$ is contained in the center of $\Gamma'$.
In other words, $\Gamma'$ is a central extension by $\mathbb{Z}^d$ 
of a finite extension of $\pi_1 (N')$.
\end{thm}

In the case of purely hyperbolic pieces, Theorem~\ref{product:thm} reduces to
Schwartz's results on the quasi-isometric rigidity of non-uniform lattices in the group of isometries of real hyperbolic space~\cite{Schwartz}.
The analogous result in the case where $N$ is closed hyperbolic has been
proved by Kleiner and Leeb \cite{klelee2}.

Once Theorems~\ref{qi-preserve:thm} and~\ref{product:thm} are established, it is not difficult to put the pieces together to obtain the following:

\begin{thm}\label{qirigidity:thm}
Let $M$ be an irreducible HDG $n$-manifold obtained by gluing the pieces
$V_i=\overline{N}_i\times T^{d_i}$, 
$i=1,\ldots, k$. Let $\Gamma$ be a group quasi-isometric
to $\pi_1 (M)$. Then either $\Gamma$ itself or a subgroup of $\Gamma$ of index two is isomorphic
to the fundamental group of a graph of groups satisfying the following conditions:
\begin{itemize}
\item
every edge group contains $\mathbb{Z}^{n-1}$ as a subgroup
of finite index;
\item
for every vertex group $\G_v$ there exist $i\in\{1,\ldots, k\}$,
a finite-sheeted covering $N'$ of $N_i$ and a finite-index subgroup
$\G'_v$ of $\G_v$ that fits into the exact sequences
$$
\xymatrix{
1\ar[r] &\mathbb{Z}^{d_i} \ar[r]^j & \Gamma_v' \ar[r] & \Delta \ar[r]& 1,\\
}
$$
$$
\xymatrix{
1\ar[r] & F \ar[r] & \Delta\ar[r] & \pi_1 (N')\ar[r] & 1 ,
}
$$
where $F$ is a finite group, and $j(\mathbb{Z}^{d_i})$ is contained
in the center of $\Gamma'_v$.
\end{itemize}
\end{thm}

\subsection{Plan of the paper}
In Section~\ref{definitions:sec} we recall the definition of quasi-isometries, we briefly discuss the fundamental Milnor-Svarc Lemma and we state a well-known characterization
of quasi-isometry among groups based on the notion of quasi-action. 
In Section~\ref{asymptotic:sec} we introduce asymptotic cones, and we discuss an asymptotic characterization of (relatively) hyperbolic groups due to Drutu and Sapir that will prove useful in the subsequent sections.
 Being the fundamental groups of manifolds which decompose into pieces, the groups we are interested in are quasi-isometric to trees of spaces, that
 are introduced in Section~\ref{trees:sec}, where their asymptotic cones are also carefully analyzed.
 We come back from homeomorphisms between asymptotic cones to quasi-isometries of the original spaces in Section~\ref{main1:sec}, where we give the proofs of Theorems~\ref{kaplee} and~\ref{qi-preserve:thm}.
 QI-rigidity of the classes of fundamental groups of irreducible non-geometric $3$-manifolds and of irreducible HDG manifolds is discussed in Section~\ref{QIrigidity:sec}.
 In Section~\ref{open} we collect some open questions.
 Finally, the appendix is devoted to the description of the behaviour of quasi-isometries with respect to 
free decompositions of groups as exhaustively illustrated in~\cite{PapaWhyte}, and to the proofs of Theorems~\ref{prime:thm} and~\ref{prime2:thm}.

\section*{Acknowledgments}
Many thanks go to Elia Fioravanti for several interesting conversations on Kapovich and Leeb's proof of the invariance of the JSJ decomposition under quasi-isometries.
The author is indebted to Yves de Cornulier for suggesting a quick proof of the quasi-isometric rigidity of $Nil$ (see Proposition~\ref{8geom:rigidity}).

\section{Quasi-isometries and quasi-actions}\label{definitions:sec}
Let $(X,d),(Y,d')$ be metric spaces and $k\geq 1$, $c\geq 0$ be real numbers. 
A (not necessarily continuous) map $f\colon X\to Y$
is a $(k,c)$-\emph{quasi-isometric embedding} if for every $p,q\in X$ the following inequalities hold:
$$
\frac{d(p,q)}{k}-c\leq d'(f(p),f(q))\leq k\cdot d(p,q)+c.
$$
Moreover, 
a $(k,c)$-quasi-isometric embedding $f$ is a $(k,c)$-\emph{quasi-isometry} if there exists
a $(k,c)$-quasi-isometric embedding $g\colon Y\to X$ such that $d'(f(g(y)),y)\leq c$, 
$d(g(f(x)),x)\leq c$ for every $x\in X$, $y\in Y$. Such a map $g$ is called a \emph{quasi-inverse}
of $f$. It is easily seen that 
a $(k,c)$-quasi-isometric embedding $f\colon X\to Y$ is a $(k',c')$-quasi-isometry for some $k'\geq 1$, 
$c'\geq 0$ if and only if its image is
$r$-dense for some $r\geq 0$, \emph{i.e.}~if every point in $Y$ is at distance at most $r$ from
some point in $f(X)$ (and in this case $k',c'$ only depend on $k,c,r$).

In the introduction we recalled the definition of Cayley graph of a group with respect to a finite set of generators, and 
it is immediate to check that different finite sets of generators for the same group define quasi-isometric 
Cayley graphs, so that every finitely generated group is endowed with a metric which is well-defined
up to quasi-isometry. 
For later reference we recall that left translations induce a well-defined isometric action of every group on any of its Cayley graphs.

The following fundamental result relates the quasi-isometry type of a group to
the quasi-isometry type of a metric space on which the group acts geometrically.
A metric space $X$ is \emph{geodesic} if every two points in it can be joined by a geodesic, i.e.~an isometric embedding of a closed interval. 
An isometric action $\Gamma\times X\to X$ of a group $\Gamma$ on a metric space $X$ is
\emph{proper} if for every bounded subset $K\subseteq X$
the set $\{g\in\Gamma\, |\, g\cdot K\cap K\neq \emptyset \}$ is finite, and \emph{cobounded}
if there exists a bounded subset $Y\subseteq X$ such that $\Gamma\cdot Y=X$ (or, equivalently, if one or equivalently every orbit of $\Gamma$ in $X$ is 
$r$-dense for some $r>0$, which may depend on the orbit). An isometric action is \emph{geometric} if it is proper and cobounded.

\begin{thm}[Milnor-Svarc Lemma]\label{milsv}
Suppose $\Gamma$ acts geometrically on a 
geodesic space $X$. Then $\Gamma$ is finitely generated and quasi-isometric to $X$, a quasi-isometry
being given by the map
$$
\psi\colon\Gamma\to X,\qquad \psi(\gamma)=\gamma(x_0),
$$
where $x_0\in X$ is any basepoint.
\end{thm} 

A proof of this result can be found e.g.~in
\cite[Chapter I.8.19]{BH}. As a corollary, if $M$ is a compact Riemannian manifold with Riemannian universal covering $\widetilde{M}$,
then the fundamental group of $M$ is quasi-isometric to $\widetilde{M}$. Another interesting consequence of Milnor-Svarc Lemma is the fact that virtual isomorphism implies quasi-isometry:

\begin{rem}\label{vi-qi}
Suppose that the groups $\G,\G'$ are such that there exists a short exact sequence
$$
\xymatrix{
1 \ar[r] & F \ar[r] & \G \ar[r]^\varphi & \G'' \ar[r] & 1\ ,
}
$$
where $\G''$ is a finite index subgroup of $\G'$, and let $X$ be any Cayley graph of $\G'$. Then every $\gamma\in \G$ isometrically acts on $X$
via the left translation by $\varphi(\gamma)$. This action is cocompact because $\G''$ has finite index in $\G'$, and proper because $F$ is finite. By Milnor-Svarc Lemma,
this implies that $\G$ is quasi-isometric to $\G'$.
\end{rem}

Thanks to Milnor-Svarc Lemma, two groups acting geometrically on the same proper geodesic space are quasi-isometric.
The converse implication does not hold: namely, there are examples of quasi-isometric groups that do not share any geometric model, i.e.~for which there exist
no geodesic space on which both groups can act geometrically. In fact, it is proved in~\cite{MSW1} that, if $p,q$ are distinct odd primes, then the groups 
$\G=\mathbb{Z}_p * \mathbb{Z}_q$ and $\G'=\mathbb{Z}_p * \mathbb{Z}_p$ are quasi-isometric (since they are virtually isomorphic to the free group on two generators) but
do not share any geometric model. Things get easier if one weakens the notion of action into the one of \emph{quasi-action}:

\begin{defn}
Suppose $(X,d)$ is a geodesic metric space, let ${\rm QI} (X)$ be the set of quasi-isometries of $X$ into itself, and let $\Gamma$ be a group. 
For $k\geq 1$, 
a $k$-\emph{quasi-action} of $\Gamma$ on $X$
is a map $h\colon \Gamma\to {\rm QI} (X)$ such that the following conditions hold:
\begin{enumerate}
\item
$h(\gamma)$ is a $(k,k)$-quasi-isometry with $k$-dense image 
for every $\gamma\in\Gamma$;
\item
$d(h(1)(x),x)\leq k$ for every $x\in X$;
\item
the composition $h(\gamma_1)\circ h(\gamma_2)$ is at distance bounded by $k$ from the quasi-isometry
$h(\gamma_1 \gamma_2)$, \emph{i.e.}
$$
d\big(h(\gamma_1 \gamma_2)(x),h(\gamma_1)(h(\gamma_2)(x))\big)\leq k\quad {\rm for\ every}\ x\in X,\ 
\gamma_1,\gamma_2\in\Gamma .
$$
\end{enumerate}
A $k$-quasi-action $h$ as above is \emph{cobounded} if one orbit of $\Gamma$ in $X$ is $k'$-dense for some $k'>0$ (or, equivalently, if every orbit of $\Gamma$ in $X$ is $k'$-dense,
for a maybe bigger $k'$), and proper if for every bounded subset $K\subseteq X$
the set $\{g\in\Gamma\, |\, g\cdot K\cap K\neq \emptyset \}$ is finite.
A quasi-action is  \emph{geometric} if it is cobounded and proper.
\end{defn}
Throughout the whole paper, by an abuse of notation, when $h$ is a quasi-action as above
we do not distinguish between $\gamma$ and $h(\gamma)$.

We have already observed that the property that the group $\G$ be quasi-isometric to $X$ is not sufficient to guarantee
that $\G$ acts geometrically on $X$. On the contrary,  we can ask for geometric \emph{quasi}-actions:

\begin{prop}\label{quasiact:prop}
Let $\varphi\colon\G\to X$ be a quasi-isometry between a group and a geodesic metric space
with quasi-inverse
$\psi\colon X\to\Gamma$. 
Then the formula
$$h(\gamma)(x)=\varphi(\gamma\cdot\psi(x))\qquad {\rm for\ every}\ x\in X $$
defines a geometric quasi-action
$h(\gamma)\colon X\to X$.
\end{prop}
\begin{proof}
The fact that $h$ is indeed a quasi-action readily follows from the fact that left translations are isometries of $\G$.
Properness and coboundedness of $h$ are easily checked.
\end{proof}

The usual proof of Milnor-Svarc's Lemma may easily be adapted to the context of quasi-actions to yield the following:

\begin{lemma}[\cite{FLS}, Lemma 1.4]\label{milsv+:lem}
Let $X$ be a geodesic space with basepoint $x_0$, and let $\Gamma$ be a group.
Let $h\colon \Gamma\to {\rm QI} (X)$ be a geometric quasi-action 
of $\Gamma$ on $X$.
Then $\Gamma$ is finitely generated and the map 
$\varphi\colon \Gamma\to X$ defined by $\varphi(\gamma)=\gamma(x_0)$ is a quasi-isometry.
\end{lemma}

\begin{cor}
 Let $\G,\G'$ be finitely generated groups. Then $\G$ is quasi-isometric to $\G'$ if and only if there exists a geodesic metric space
 $X$ such that both $\G$ and $\G'$ geometrically quasi-act on $X$.
\end{cor}
\begin{proof}
 The ``if'' implication follows from Lemma~\ref{milsv+:lem}. On the other hand, if $\G$ is quasi-isometric to $\G'$ and $X$ is a fixed Cayley graph for $\G'$, then
 $\G$  geometrically quasi-acts on $X$ by Proposition~\ref{quasiact:prop}, and $G'$ geometrically acts on $X'$ by left translations.
\end{proof}

\section{Asymptotic cones}\label{asymptotic:sec}
Roughly speaking, 
the asymptotic cone of a metric space gives a picture of the metric
space as ``seen from infinitely far away''. It was introduced by Gromov in~\cite{Gro-nilpotent},
and formally defined in~\cite{wilkie}. Being uninfluenced by finite errors, the asymptotic cone is a quasi-isometry invariant, and turns 
discrete objects into continuous spaces. Citing Gromov, ``This space [the Cayley graph of a group $\Gamma$] may appear boring and uneventful to a geometer's eye since it is discrete and the traditional
local (e.g.~topological and infinitesimal) machinery does not run in $\G$. To regain the geometric perspective one has to change his/her position and move the observation point far away from $\G$. Then the metric in $\G$ seen from the distance $d$ becomes the original distance
divided by $d$ and for $d\to\infty$ the points in $\G$ coalesce into a connected continuous solid unity which occupies the visual horizon without any gaps or holes and fills our geometer's hearth with joy''~\cite{Gromov-asymptotic}.

Gromov himself provided a characterization of word hyperbolicity in terms of asymptotic cones~\cite{Gromov-hyperbolic, Gromov-asymptotic} (see also~\cite{Drutu1} 
and~\cite{FS}). 
Also \emph{relative} hyperbolic groups admit a neat (and very useful) characterization via asymptotic cones~\cite{DS}. It should maybe worth mentioning that, while having nice metric 
properties (e.g., asymptotic cones of Cayley graphs of groups are complete, homogeneous and geodesic), asymptotic cones are quite wild from the topological point of view. 
They are often not locally compact (for example, a group is virtually nilpotent if and only if all its asymptotic cones are locally compact~\cite{Gro-nilpotent, Drutu1} if and only if one of its asymptotic cone
is locally compact~\cite{Sapirnew, Hru, sisto2}), 
and their homotopy type can be quite complicated: Gromov himself~\cite{Gromov-asymptotic} conjectured that the fundamental group of any asymptotic cone of any group
either is simply connected, or has an uncountable fundamental group. While this conjecture eventually turned out to be false in general~\cite{OS2}, usually non-simply connected asymptotic cones of groups are rather complicated: Erschler and Osin showed
that every countable group is a subgroup of the fundamental group of an asymptotic cone of a finitely generated
group~\cite{EO} (this result was then sharpened in~\cite{DS}, where it is shown that, for every countable group $G$, there exists an
asymptotic cone of a finitely generated group whose fundamental group is the free product of uncountably
many copies of $C$). Moreover, Gromov's conjectural dichotomy was proved to hold in several cases
(see e.g.~\cite{Kent1, CK} for results in this direction).

Let us recall the definition of asymptotic cone of a space.
A \emph{filter} on $\mathbb{N}$ is a set $\omega\subseteq \mathcal{P}(\mathbb{N})$ satisfying the following conditions:
\begin{enumerate}
\item
$\emptyset\notin \omega$;
\item
$A,B\in \omega\ \Longrightarrow\  A\cap B\in \omega$;
\item 
$A\in \omega,\ B\supseteq A\ \Longrightarrow\ B\in\omega$.
\end{enumerate}
For example, the set of complements of finite subsets of $\mathbb{N}$
is a filter on $\mathbb{N}$, known as the \emph{Fr\'echet filter} on $\mathbb{N}$.

A filter $\omega$ is a \emph{ultrafilter} if for every $A\subseteq\mathbb{N}$
we have either $A\in\omega$ or $A^c\in\omega$, where $A^c :=\mathbb{N}\setminus A$.
For example, fixing an element $a\in \mathbb{N}$, we can take the associated
\emph{principal ultrafilter} to consist of all subsets of $\mathbb{R}bb N$ which 
contain $a$. An ultrafilter is \emph{non-principal} if it does not contain any 
finite subset of $\mathbb{N}$. 

It is readily seen that a filter is an ultrafilter if and only if it is maximal
with respect to inclusion. Moreover, an easy application of Zorn's Lemma
shows that any filter is contained in a maximal one. Thus, non-principal
ultrafilters exist (just take any maximal filter containing the Fr\'echet filter). 

From this point on, let us fix a non-principal ultrafilter $\omega$ on $\mathbb{N}$.
As usual, we say that a statement $\mathcal{P}_i$ depending on $i\in\mathbb{N}$ holds 
$\omega$-a.e.~if the set of indices such that $\mathcal{P}_i$ holds belongs to $\omega$. 
If $X$ is a topological space, and $(x_i)\subseteq X$ is a  
sequence in $X$, 
we say that $\omega$-$\lim x_i =x_\infty$ if 
$x_i\in U$ $\omega$-a.e. for every neighbourhood $U$ of $x_\infty$.
When $X$ is Hausdorff, an $\omega$-limit of a sequence, if it
exists, is unique. Moreover, any sequence in any compact space admits
an $\omega$-limit. For example, any sequence $(a_i)$ in
$[0,+\infty]$ admits a unique $\omega$-limit.

Now let $(X_i,x_i,d_i)$, $i\in\mathbb{N}$, be a sequence of pointed metric spaces. 
Let $\mathcal{C}$ be the set of sequences $(y_i), y_i\in X_i$,
such that $\omega$-$\lim {d_i(x_i,y_i)}<+\infty$, and consider the 
following equivalence
relation on $\mathcal{C}$:
$$
(y_i)\sim (z_i)\quad \Longleftrightarrow \quad
\omega\text{-}\lim {d_i(y_i,z_i)}=0.
$$
We set $\omega$-$\lim (X_i,x_i,d_i)=\mathcal{C}/_\sim$, and we endow 
$\omega$-$\lim (X_i,x_i,d_i)$ with the well-defined distance given by
$d_\omega \big([(y_i)],[(z_i)]\big)=\omega$-$\lim {d_i(y_i,z_i)}$. 
The pointed metric space $(\omega$-$\lim (X_i,x_i,d_i),d_\omega)$ is called the 
\emph{$\omega$-limit} of the pointed metric spaces $X_i$.

Let $(X,d)$ be a metric space, $(x_i)\subseteq X$ a sequence of
base-points, and $(r_i)\subset \mathbb{R}^+$ a sequence of rescaling factors
diverging to infinity. We introduce the notation $(X_\omega ((x_i),(r_i)),d_\omega):=
\omega$-$\lim (X_i,x_i,{d}/{r_i})$.

\begin{defn}\label{cone:def}
The metric space $\big(X_\omega\big((x_i),(r_i)\big),d_\omega\big)$ 
is the \emph{asymptotic cone}
of $X$ with respect to the ultrafilter $\omega$, 
the basepoints $(x_i)$ and the rescaling factors $(r_i)$. 
For conciseness, we will occasionally just write $X_\omega\big((x_i),(r_i)\big)$
for the asymptotic cone, the distance being implicitly understood to
be $d_\omega$, or even $X_\omega$ when basepoints and rescaling factors are fixed. 
\end{defn}

If $\omega$ is fixed and $(a_i)\subseteq \mathbb{R}$ is any sequence, we say that $(a_i)$
is $o(r_i)$ (resp.~$O(r_i)$) if $\omega$-$\lim {a_i}/{r_i}=0$ (resp.~$\omega$-$\lim {|a_i|}/{r_i}<\infty$).
Let $(x_i)\subseteq X$, $(r_i)\subseteq \mathbb{R}$ be fixed sequences of basepoints
and rescaling factors, and set $X_\omega=(X_\omega((x_i),(r_i)),d_\omega)$.
Sequences of subsets in $X$ give rise
to subsets of $X_\omega$: if for every $i\in\mathbb{N}$ we are given a subset
$\emptyset\neq A_i\subseteq X$, we set
$$
\omega\text{-}\lim A_i=\{[(p_i)]\in X_\omega\, |\, p_i\in A_i\ {\rm for\ every}\ i\in\mathbb{N}\}.
$$
It is easily seen that for any choice of the $A_i$'s, the set $\omega$-$\lim A_i$
is closed in $X_\omega$. Moreover, $\omega$-$\lim A_i\neq\emptyset$ if and only
if the sequence $(d(x_i,A_i))$ is $O(r_i)$.

It is usually quite difficult to describe asymptotic cones of spaces. Nevertheless, in some cases the situation is clear: for example, if the metric of $X$ is homogeneous and scale-invariant (meaning that
$(X,d)$ is isometric to $(X,d/r)$ for every $r>0$) then every asymptotic cone of $X$ is isometric to $X$ (this is the case e.g.~for $X=\mathbb{R}^n$). We will see below that also asymptotic cones of hyperbolic spaces are easily understood.

\subsection{Asymptotic cones and quasi-isometries}
Quasi-isometries are just maps that are bi-Lipschitz up to a finite error. Therefore, it is not surprising that
they asymptotically define
bi-Lipschitz homeomorphisms. In fact, it is an easy exercise to check that,
once a non-principal ultrafilter $\omega$ is fixed,
if $(X_i,x_i,d_i)$
$(Y_i,y_i,d'_i)$, $i\in\mathbb{N}$ are pointed metric spaces, 
$(r_i)\subset\mathbb{R}$ is a sequence of rescaling factors, and $f_i\colon X_i\to Y_i$ are $(k_i,c_i)$-quasi-isometries such that
$k=\omega$-$\lim k_i<\infty$, $c_i=o(r_i)$ and $d_i' (f_i (y_i),x_i)=O(r_i)$, 
then the formula $[(p_i)]\mapsto
[f_i(p_i)]$ provides a well-defined $k$-bi-Lipschitz embedding $f_\omega\colon \omega$-$\lim (X_i,x_i,d_i/r_i)\to
\omega-\lim (Y_i,y_i,d'_i/r_i)$. 

As a corollary, quasi-homogeneous spaces (i.e.~metric spaces whose isometry group admits $r$-dense orbits for some $r>0$) have homogeneous
asymptotic cones, whose isometry type does not depend on the choice of basepoints. Moreover,
quasi-isometric metric spaces have bi-Lipschitz homeomorphic asymptotic cones. In particular, quasi-isometric groups have homogeneous bi-Lipschitz homeomorphic
asymptotic cones.

Actually, the last sentence of the previous paragraph should be stated more precisely as follows: once a non-principal ultrafilter
and a sequence of rescaling factors are fixed, any quasi-isometry between two groups induces a bi-Lipschitz homeomorphism
between their asymptotic cones. Here and in what follows we will not focus on the dependence of asymptotic cones
on the choice of ultrafilters and rescaling factors. In fact, in many applications (and in all the arguments described in this paper) the role played
by such choices is very limited (but see the proof of Theorem~\ref{wall:thm}). Let us just mention here that, answering to a question of Gromov~\cite{Gromov-hyperbolic}, Thomas and 
Velickovic exhibited a finitely generated group with two non-homeomorphic asymptotic cones~\cite{TV}
(the first finitely presented example of such a phenomenon is due to Ol'shanskii and Sapir~\cite{OS}).

\subsection{Asymptotic cones of (relatively) hyperbolic groups}\label{treegr:subsec}
Recall that a geodesic metric space $X$ is $\delta$-\emph{hyperbolic}, where $\delta$ is a non-negative constant, if every side of every geodesic triangle in $X$
is contained in the $\delta$-neighbourhood of the union of the other two sides (in this case one says that the triangle is $\delta$-fine). A space is hyperbolic if
it is $\delta$-hyperbolic for some $\delta\geq 0$. A fundamental (and not completely trivial) fact is that being hyperbolic is a quasi-isometry invariant, so it makes sense
to speak of hyperbolic groups. 
A \emph{real tree} is a $0$-hyperbolic geodesic metric space. Simplicial trees are real trees, but many wilder examples of real trees can be easily constructed (and 
naturally arise as asymptotic cones of groups!).

Every pinched negatively curved simply connected Riemannian manifold is hyperbolic, so by Milnor-Svarc Lemma the fundamental group of every compact negatively curved Riemannian
manifold is hyperbolic. In fact, such groups were probably the motivating examples that lead Gromov to the definition of hyperbolicity. Of course, 
by dividing the metric by a factor $r_n$ we can turn a $\delta$-hyperbolic space into a $\delta/r_n$-hyperbolic space, and this implies that
any asymptotic cone of a hyperbolic space is a real tree (after one proves that geodesics in $X_\omega$  are $\omega$-limits of geodesics in $X$: this is false for generic geodesic
spaces,
and true if $X$ is hyperbolic). More precisely, if $\G$ is a non-virtually cyclic hyperbolic group, then every asymptotic cone of $\G$
is a homogeneous real tree each point of which is a branching point
with uncountably many branches. In particular, all the asymptotic cones of all non-virtually cyclic hyperbolic groups are pairwise isometric~\cite{DP}. 

As we mentioned above, being asymptotically $0$-hyperbolic indeed characterizes hyperbolic spaces: if every asymptotic cone of $X$ is a real tree, then $X$ 
is hyperbolic~\cite{Gromov-hyperbolic, Gromov-asymptotic, Drutu1,FS}. This implies, for example,
the non-trivial fact that if every triangle of diameter $D$ in $X$ is $\delta(D)$-fine, where $\delta(D)$ grows sublinearly with respect to $D$, then $X$ is hyperbolic (so $\delta$ is 
bounded). It is maybe worth mentioning that indeed one needs \emph{all} the asymptotic cones of $X$ to be real trees in order to get hyperbolicity of $X$:
groups having at least one asymptotic cone which is a real tree are known as \emph{lacunary hyperbolic}, and were studied in~\cite{OS}; note however that a result of Gromov~\cite{Gromov-hyperbolic}
ensures that a finitely presented lacunary hyperbolic group is indeed hyperbolic.

Just as hyperbolic groups generalize fundamental groups of negatively curved compact manifolds, the notion of \emph{relatively hyperbolic} group was introduced by Gromov
to extend the class of fundamental groups of pinched negatively curved complete  manifolds of finite volume (see e.g.~the fundamental paper of Farb~\cite{Farb2}). Here we will
define relative hyperbolic groups by means of a later asymptotic characterization which is due to Dru\c{t}u and Sapir~\cite{DS}.
 If $X$ is a set, then we denote by $|X|$ the cardinality of $X$.

\begin{defn}
A geodesic metric space $X$ is said to be \emph{tree-graded} 
with respect to a collection of closed subsets $\{P_i\}_{i\in I}$, 
called \emph{pieces}, if
\begin{enumerate}
\item
$\bigcup P_i=X$,
\item
$|P_i\cap P_j|\leq 1$ if $i\neq j$,
\item
any simple geodesic triangle in $X$  is contained in a single piece.
\end{enumerate}
\end{defn}

In fact, it is proved in~\cite[Lemma 2.15]{DS} that if 
 $A$ is a path-connected subset of a tree-graded space $Y$ and $A$ has no cut-points, then $A$ is contained in a piece.
 In order to avoid trivialities, we will always assume that no piece is equal to the whole tree-graded space. 

\begin{defn}\label{as-tree:def}
A geodesic metric space $X$ is \emph{asymptotically tree-graded} with respect to a collection of subsets
$A=\{H_i\}_{i\in I}$ if the following conditions hold: 
\begin{enumerate}
\item 
for each choice of basepoints $(x_i)\subseteq X$ and rescaling factors $(r_i)$, 
the associated asymptotic cone $X_\omega=X_\omega((x_i),(r_i))$ is tree-graded
with respect to the collection of subsets 
$\mathcal{P}=\{\omega$-$\lim H_{i(n)}\, |\, H_{i(n)}\in A\}$, and
\item if $\omega$-$\lim H_{i(n)}= 
\omega$-$\lim H_{j(n)}$, where $i(n),j(n)\in I$, then $i(n)=j(n)$ $\omega$-a.e.
\end{enumerate}
\end{defn} 

We are now ready to give the definition of relatively hyperbolic group.

\begin{defn}\label{drusap:def}
 Let $\G$ be a group and let $\{H_1,\ldots,H_k\}$ be a finite collection of proper infinite subgroups of $\G$. Then 
 $\G$ is hyperbolic relative to $\{H_1,\ldots,H_k\}$ if it is asymptotically tree-graded with respect
 to the family of all left cosets of the $H_i$'s. If this is the case, the $H_i$'s are called \emph{peripheral subgroups} of $\G$.
 \end{defn}

In some sense, the asymptotic cone of a relatively hyperbolic group is obtained from a real tree by blowing up points into pieces where the non-tree-like behaviour
is concentrated.

Definition~\ref{drusap:def} implies that a group admitting an asymptotic cone without cut-points cannot be relatively hyperbolic (such a group is said to be \emph{unconstricted})~\cite{DS}.
In fact, the following stronger result holds:

\begin{prop}[\cite{DS,BDM}]\label{unconstricted}
 If $\G$ is unconstricted, then every $(k,c)$-quasi-isometric embedding of $\G$ into a relatively hyperbolic group lies in a $D$-neighbourhood of a coset of a peripheral subgroup,
where $D$ only depends on $k,c$ (and on the geometry of the relatively hyperbolic group).
\end{prop}

Recall that the geometric pieces
of the manifolds we are interested in (i.e.~irreducible $3$-manifolds and HDG manifolds) each consist of the product of a cusped hyperbolic manifold $N$
with a torus (where the cusped manifold may be just a surface in the case of pieces modeled on $\mathbb{H}^2\times\mathbb{R}$, and the torus may be reduced to a circle or to a point).
The fundamental groups of cusped hyperbolic manifolds are hyperbolic relative to the cusp subgroups (see e.g.~\cite{Farb2} for a much more general result dealing with variable negative curvature). 
As usual, in order to better understand the geometry of the group one would like to make use of Milnor-Svarc Lemma,
and to this aim we need to replace our cusped $n$-manifold $N$ with the one obtained by truncating the cusps, which will be denoted by $\overline{N}$. The universal covering of $\overline{N}$
is  the complement in hyperbolic space of an equivariant collection of disjoint horoballs whose union is just the preimage of the cusps under the universal covering map $\mathbb{H}^n\to N$. 
Such a space $B$ is usually called a \emph{neutered space}, and will be always endowed with the path metric induced by its Riemannian structure.
Its boundary is given by an equivariant collection of disjoint horospheres, each of which is totally geodesic (in a metric sense) and flat.
Now the quasi-isometry provided by Milnor-Svarc Lemma sends each coset of a peripheral subgroups of $\pi_1(N)$
into a component of $\partial B$ (this is true as stated if $N$ has one cusp and the basepoint of $B$ is chosen on $\partial B$, and true only up to finite errors
otherwise). Together with the invariance of asymptotic cones with respect to quasi-isometries, we may conclude that $B$ is asymptotically tree--graded with respect to
the family of the connected components of $\partial B$.

We recall that an \emph{$n$-flat} is a totally geodesic embedded subspace isometric to $\mathbb{R}^n$, and that whenever $X,Y$ are metric spaces we understand
that $X\times Y$ is endowed with the induced $\ell^2$-metric (which coincides with the usual product metric in the case of Riemannian manifolds).
Since the asymptotic cone of a product is the product of the asymptotic cones of the factors,
the following result summarizes the previous discussion:

\begin{prop}\label{omegachambers}
 Let $M=\overline{N}\times (S^1)^k$, where $\overline{N}$ is a complete finite volume hyperbolic $n$-manifold with cusps removed and $(S^1)^k$ is a flat $k$-dimensional torus, let $X$ be the metric universal covering of $M$
and let $X_\omega$ be any asymptotic cone of $X$. Then $X_\omega$ is isometric to $Y\times \mathbb{R}^k$, where $Y$ is a tree-graded space and every piece of $Y$ is an $(n-1)$-flat. In particular, if $\dim N=2$, then
$X_\omega$ is isometric to the product of a real tree with $\mathbb{R}^k$.
 \end{prop}

 The last sentence of the previous proposition can be explained in two ways. If $\dim \overline{N}=2$, then $\pi_1(\overline{N})$ is free, whence hyperbolic. On the other hand,
 $\pi_1(\overline{N})$ is also hyperbolic relative to the peripheral subgroups, which are isomorphic to $\mathbb{Z}$, so its asymptotic cone is tree-graded with pieces isometric to lines;
 and it is easy to see that a tree-graded space with pieces that are real trees is itself a real tree.
 
 \begin{cor}\label{aaa}
 Let $M_1,M_2$ be a Seifert $3$-manifold with non-empty boundary and a non-compact complete finite-volume hyperbolic $3$-manifold, respectively. Then
 $\pi_1(M_1)$ is not quasi-isometric to $\pi_1(M_2)$.
 \end{cor}
\begin{proof}
By Proposition~\ref{omegachambers}, it is sufficient to observe that the product  of a real tree with the real line does not contain cut-points, while
any tree-graded space does.
\end{proof}

  Of course free abelian groups are unconstricted, so Proposition~\ref{unconstricted} (together with Milnor-Svarc Lemma) readily implies the following:
 
 \begin{prop}\label{neutered1}
  Let $B\subseteq\mathbb{H}^n$ be an $n$-dimensional neutered space as above. Then the image of any $(k,c)$-quasi-isometric embedding of $\mathbb{R}^{n-1}$
  into $B$ lies in the $D$-neighbourhood of a component of $\partial B$, where $D$ only depends on $(k,c)$ (and on the geometry of $B$).
 \end{prop}
 
 Proposition~\ref{neutered1} was first proved by Schwartz in~\cite{Schwartz}, where it provided one of the key steps in the proof of the quasi-isometric rigidity of non-uniform lattices in the isometry group
 of real hyperbolic $n$-space, $n\geq 3$.
 In fact, Proposition~\ref{neutered1} implies that any quasi-isometry  $f\colon B\to B'$ between neutered spaces must coarsely send $\partial B$ into $\partial B'$. As a consequence, the map $f$ can be extended to
a quasi-isometry $\overline{f}$ of the whole of $\mathbb{H}^n$. An additional argument exploiting the fact that $\overline{f}$ sends an equivariant family of horospheres into another equivariant family of horospheres 
and involving a fine analysis  of the trace of $\overline f$ on $\partial \mathbb{H}^n$ allows to conclude that $\overline f$ is uniformly close to an isometry $g$. Then it is 
not difficult to show 
that $g$ almost conjugates the isometry group of $B$
 into the isometry group of $B'$, so that these isometry groups (which are virtually isomorphic to the fundamental groups of any compact quotient of $B,B'$, respectively) are virtually isomorphic
 (see the discussion in Subsection~\ref{subseccita}).
 This argument should somewhat clarify the statement we made in the introduction, according to which 
the existence of patterns with a distinguished geometric behaviour in a space $X$ usually imposes some rigidity on geometric maps of $X$ into itself. In the case just described the pattern is given by a family of flats, whose 
intrinsic geometry already make them detectable in the hyperbolic context where they lie. In other situations the single objects of the pattern do not 
enjoy peculiar intrinsic features. This is the case,
for example, for the boundary components of the universal covering of a compact hyperbolic $n$-manifold with geodesic boundary. Such boundary components are themselves hyperbolic hyperplanes in $\mathbb{H}^n$
(in particular, their asymptotic cones are \emph{not} unconstricted), so in order to prove that they are recognized by quasi-isometries one cannot rely on general results
on relatively hyperbolic groups. Nevertheless, the general strategy described by Schwartz still applies to get quasi-isometric rigidity~\cite{Fri}, 
once one replaces Proposition~\ref{unconstricted} with an argument making use of a suitable notion of \emph{coarse separation}. Such a notion was also at the hearth of Schwartz's original argument for
Proposition~\ref{neutered1}.
 
 Schwartz's strategy to prove QI-rigidity of non-uniform lattices may be in fact pursued also in the more general context of relatively hyperbolic groups. In fact, building on Proposition~\ref{unconstricted} it is possible to show that,
 if $\G$ is hyperbolic relative to $\{H_1,\ldots,H_k\}$, where each $H_i$ is unconstricted, and $\G'$ is quasi-isometric to $\G$, then
 also $\G'$ is hyperbolic relative to a a finite collection of subgroups each of which is quasi-isometric to one of the $H_i$'s~\cite{DS,BDM}. In fact,
 the hypothesis on the peripheral subgroups of $\G$ being unconstricted may be weakened into the request that each $H_i$ be not relatively hyperbolic with respect to any finite collection of proper subgroups.
 However, this is not sufficient to conclude that the class of relatively hyperbolic groups is quasi-isometrically rigid, since there exist relatively hyperbolic groups that have no list of peripheral subgroups composed uniquely
 of non-relatively hyperbolic groups (indeed, the inaccessible group constructed by Dunwoody~\cite{Dun2} is also an example of such a relatively hyperbolic group, see~\cite{BDM}; see
 the appendix for a brief discussion of (in)accessibility of groups). Nevertheless, it eventually turns out that the
 whole class of relatively hyperbolic groups is QI-rigid:
 
 \begin{thm}[\cite{Drutu3}]
  The $\G$ be a group hyperbolic relative to a family of subgroups $\{H_1,\ldots,H_k\}$, and let $\G'$ be a group quasi-isometric to $\G$. Then $\G'$
  is hyperbolic relative to a family $\{H_1',\ldots,H_s'\}$ of subgroups, where each $H_i'$ can be quasi-isometrically embedded
  in $H_j$ for some $j=j(i)$.
 \end{thm}

 Before concluding the section, let us come back to what asymptotic cones can and cannot distinguish. We have already seen that
 the asymptotic cones of hyperbolic groups are all quasi-isometric to each other, so 
they cannot be exploited to distinguish non-quasi-isometric hyperbolic groups. However, asymptotic cones can tell apart the quasi-isometry classes of the eight $3$-dimensional geometries:

\begin{prop}\label{8models}
With the exception of the case of $\mathbb{H}^2\times \mathbb{R}$ and $\widetilde{SL_2}$, the eight $3$-dimensional geometries can be distinguished by looking at their asymptotic cones.
In particular, with the exception of the case of $\mathbb{H}^2\times \mathbb{R}$ and $\widetilde{SL_2}$, distinct geometries are not quasi-isometric.
\end{prop}
\begin{proof}
Every asymptotic cone of the sphere is a point. The Euclidean $3$-space and $Nil$ are the unique geometries such that each of their asymptotic cones is unbounded and
locally compact.
Moreover, both these spaces admit a nontrivial 1-parameter group of non-isometric similarities, so they are isometric to their asymptotic cones. However, it is not diffult to show that 
$\mathbb{R}^3$ is not bi-Lipschitz equivalent to $Nil$ (in fact, they are not even quasi-isometric, since they have different growth rates). 
Every asymptotic cone of $S^2\times\mathbb{R}$, of $\mathbb{H}^3$ and of $\mathbb{H}^2\times\mathbb{R}$ is isometric to
the real line, to the homogeneous real tree $T$ with uncountably many branches at every point, and 
to the product $T\times\mathbb{R}$, respectively. In particular, they are all simply connected and pairwise non-bi-Lipschitz equivalent. Finally, $Sol$ is the unique $3$-dimensional
geometry admitting non-simply connected asymptotic cones~\cite{Burillo}.
\end{proof}

 \section{Trees of spaces and their asymptotic cones}\label{trees:sec}
 Let now $M$ be either an irreducible non-geometric $3$-manifold, or an irreducible $n$-dimensional HDG manifold, $n\geq 3$. For the sake of simplicity, in the case when $M$ is not a HDG manifold,
 henceforth we will assume that $M$ is closed.
 It follows by the very definitions that $M$ decomposes into pieces of the form $\overline{N}\times T^k$, where $\overline{N}$ is a truncated hyperbolic manifold and
 $T^k$ is a $k$-dimensional torus. Moreover, $\dim \overline {N}\geq 3$ unless $\dim M=3$, while $k$ may be any non-negative integer. We put on $M$ a Riemannian metric (in the $3$-dimensional case, we will carefully choose one   
 later), and we denote by $X$ the Riemannian universal covering of $M$. The decomposition of $M$ into pieces lifts to a decomposition of $X$ into \emph{chambers}. Since
 the fundamental group of each piece injects into $\pi_1(M)$, every chamber is the universal covering of a geometric piece of $M$. Chambers are adjacent along \emph{walls}, where a wall is the lift to $X$ of a torus
 (or Klein bottle)
 of the decomposition of $M$. This decomposition of $X$ can be encoded by a graph $T$ whose vertices (resp.~edges) correspond to chambers (resp.~walls), and the edge corresponding to the wall $W$ joins the vertices corresponding to
 the chambers adjacent to $W$. Moreover, each chamber is foliated by \emph{fibers}, which are just the lifts to $X$ of the fibers of the pieces of $M$ (where fibers of a purely hyperbolic piece are understood to be points).
  Since $X$ is simply connected, the graph $T$ is a (simplicial, non-locally finite) tree. In fact, the decomposition of $M$ into pieces realizes $\pi_1(M)$ as the fundamental group of a graph of groups
 having the fundamental groups of the pieces  as vertex groups and the fundamental pieces of the splitting tori (or Klein bottles)  as edge groups. The tree $T$ is exactly the Bass-Serre tree associated to this graph of groups
 (see~\cite{serre,SW} for the definition and the topological interpretation of the fundamental group of a graph of groups). 

 By Milnor-Svarc Lemma, Theorems~\ref{kaplee} and~\ref{qi-preserve:thm} readily follow from the following:
 
 \begin{thm}\label{main1:thm}
  Let $M,M'$ be either irreducible non-geometric graph manifolds or irreducible HDG manifolds, and denote by $X,X'$ the universal coverings of $M,M'$ respectively. Let also
  $f\colon X\to X'$ be a quasi-isometry. Then for every chamber $C$ of $X$ there exists a chamber $C'$ of $X'$ such that
  $f(C)$ lies within finite Hausdorff distance from $C'$ (as a consequence, $C$ is quasi-isometric to $C'$). Moreover, $f$ preserves the structures of $X,X'$ as trees of spaces,
  i.e.~it induces an isomorphism between the trees encoding the decomposition of $X,X'$ into chambers and walls.
 \end{thm}

 We would like to study the coarse geometry of $X$ starting from what we know about the coarse geometry of its chambers. This strategy can be 
 more easily pursued provided that chambers are quasi-isometrically embedded in $X$.
We first observe that, if the manifold $M$ carries a Riemannian
metric of non-positive curvature, then by the Flat Torus Theorem (see e.g.~\cite{BH})
the decomposition of $M$ into pieces can be realized geometrically by cutting along totally geodesic
embedded flat tori (and Klein bottles in the $3$-dimensional case). This readily implies that walls and chambers (endowed with their intrinsic path metric) are isometrically embedded in $X$.
Now it turns out that non-geometric irreducible $3$-manifolds ``generically''
 admit metrics of non-positive curvature: namely, this is true for every manifold containing at least one hyperbolic piece~\cite{leeb}, while it can fail for
 graph manifolds (Buyalo and Svetlov \cite{BS} have a complete criterion for deciding whether or not a $3$-dimensional graph  manifold supports a 
non-positively curved Riemannian metric). 
Nevertheless, it is proved in~\cite{kaplee} that  for every non-geometric irreducible $3$-manifold $M$, there exists a 
non-positively curved non-geometric irreducible $3$-manifold $M'$ 
such that the universal covers $X$ of $M$ and $X'$ of $M'$
are bi-Lipschitz homeomorphic by a homeomorphism which
preserves their structures as trees of spaces (in particular, walls and chambers are quasi-isometrically embedded in $X$). In fact,
thanks to this result it is not restrictive (with respect to our purposes) 
to consider only non-positively curved non-geometric irreducible $3$-manifolds.

In the higher dimensional case it is not possible to require non-positive curvature. Indeed, it is easy to construct HDG manifolds whose fundamental group cannot be quasi-isometric to the fundamental
group of any non-positively curved closed Riemannian manifold~\cite[Remark 2.23]{FLS}. 
At least from the coarse geometric point of view, things get a bit better when restricting to irreducible manifolds:

\begin{thm}[Theorem 0.16 in~\cite{FLS}]
 Suppose that $M$ is an irreducible HDG manifold with universal covering $X$. Then walls and chambers are quasi-isometrically embedded in $X$.
\end{thm}

Nevertheless, it is possible to construct irreducible HDG manifolds which do not support any non-positively curved metric~\cite[Theorem 0.20]{FLS}.
It is still an open question whether the fundamental group of every irreducible HDG manifold is quasi-isometric to
the fundamental group of a non-positively curved HDG manifold.

 \subsection{Irreducible $3$-manifolds versus HDG manifolds}\label{versus:sub}
 If $C$ is any chamber of the tree of spaces $X$ introduced above, then the walls of
$C$ are $r$-dense in $C$. Using this fact, it is easy to realize that 
 the key step in the proof of
 Theorem~\ref{main1:thm} consists in showing that quasi-isometries must preserve walls (and send walls that do not separate each other to walls that do not separate each other).
 In the case of irreducible HDG $n$-manifolds, using that the codimension of the fibers is big enough  one can show that walls are the unique $(n-1)$-dimensional
quasi-isometrically embedded copies of $\mathbb{R}^{n-1}$ in $X$ (see Theorem~\ref{wall:thm}), and this provides the key step towards Theorem~\ref{main1:thm}
in that case. On the contrary, chambers corresponding to Seifert pieces of irreducible $3$-manifolds contain a lot of $2$-dimensional flats coming from the lifts to $X$
of the product of  closed simple geodesics in the $2$-dimensional base with the $S^1$-fiber. As a consequence, some more work is needed to provide a quasi-isometric characterization
of walls. On the positive side, in the $3$-dimensional case one can use the peculiar features of non-positive curvature, that are not available in higher dimensions.

\subsection{Asymptotic cones of trees of spaces}
Let $M$ be either a  $3$-dimensional irreducible non-geometric $3$-manifold or an irreducible HDG manifold. For the sake of simplicity, in the first case we also assume that $M$ is closed, and endowed with a non-positively curved metric
(see the previous section). We have seen that the universal covering of $M$ decomposes as a tree of spaces into chambers separated by walls. Moreover, chambers and walls are quasi-isometrically embedded in $X$ (and even isometrically embedded
when $\dim M=3$). 

Let $\omega$ be a fixed non-principal ultrafilter on $\mathbb{N}$, let $(x_i)\subseteq X$,
$(r_i)\subseteq \mathbb{R}$ be fixed sequences of basepoints and rescaling factors,
and set $X_\omega=(X_\omega, (x_i),(r_i))$.

\begin{defn}
An $\omega$-\emph{chamber} (resp.~$\omega$-wall, $\omega$-fiber)
in $X_\omega$ is a subset $Y_\omega\subseteq X_\omega$
of the form $Y_\omega=\omega\text{-}\lim Y^i$,
where each $Y^i\subseteq X$ is a chamber (resp.~a wall, a fiber). 
An $\omega$-wall $W_\omega=\omega{\textrm -}\lim W_i$ is a \emph{boundary}
(resp.~\emph{internal}) $\omega$-wall if $W_i$ is a boundary (resp.~internal)
wall $\omega$-a.e.
\end{defn}

The decomposition of a tree-graded space into its chambers induces a 
decomposition of $X_\omega$ into its $\omega$-chambers. 
Indeed, since
a constant $k$ exists such that each point of $X$ has distance
at most $k$ from some wall, every point of $X_\omega$ lies in some $\omega$-wall.
Recall that, in a tree-graded space, subspaces
homeomorphic to Euclidean spaces of dimension bigger than one (which do not have cut-points) are contained in pieces.
We would like to prove that
a similar phenomenon occurs in our context, and this is indeed the case
when dealing with HDG manifolds. In fact, in that case
$\omega$-walls can be characterized as the only subspaces of
$X_\omega$ which are bi-Lipschitz homeomorphic to $\mathbb{R}^{n-1}$, where $n=\dim M$ (see Proposition~\ref{wall:char:prop}). As a consequence, 
every bi-Lipschitz homeomorphism of $X_\omega$ preserves the decomposition
of $X_\omega$ into $\omega$-walls. Together
with an argument which allows us to recover quasi-isometries
of the original spaces from bi-Lipschitz homeomorphisms of asymptotic cones, 
this will imply
Theorem~\ref{main1:thm} in the case of HDG manifolds. As anticipated above, extra work is needed in the case of $3$-manifolds, where the presence of Seifert pieces
implies the existence of $2$-dimensional flats in $X_\omega$ that are not $\omega$-walls.

We begin by collecting some facts that will prove useful for the proof of Theorem~\ref{main1:thm}. Henceforth we denote by $n\geq 3$ the dimension
of $M$ (and of its unviersal covering $X$).
Being a quasi-isometric embedding, the inclusion of any chamber in $X$ induces a bi-Lipschitz embedding of the asymptotic cone of the chamber into $X_\omega$. Therefore, from Proposition~\ref{omegachambers}
we deduce that 
for any $\omega$-chamber $C_\omega$ there exists a bi-Lipschitz homeomorphism
$\varphi\colon C_\omega\to Y\times 
\mathbb{R}^{l}$, where $Y$ is a tree-graded space whose pieces are bi-Lipschitz homeomorphic
to $\mathbb{R}^{n-l-1}$, such that the following conditions hold:
\begin{enumerate}
 \item 
For every $p\in Y$, the subset
$\varphi^{-1}(\{p\}\times \mathbb{R}^l)$ is an $\omega$-fiber of $X_\omega$.
\item
For every piece $P$ of $Y$, the set
$\varphi^{-1}(P\times \mathbb{R}^l)$ is an $\omega$-wall of $X_\omega$. 
\end{enumerate}
In particular, 
every $\omega$-wall of $X$ is bi-Lipschitz homeomorphic to
$\mathbb{R}^{n-1}$, and
every $\omega$-fiber of $X$ is bi-Lipschitz homeomorphic to
$\mathbb{R}^{h}$ for some  $h\leq n-3$.
A \emph{fiber of $C_\omega$} is an $\omega$-fiber 
of $X_\omega$ of the form described
in item (1) above.
A \emph{wall of $C_\omega$} is an $\omega$-wall of $X_\omega$ of the form described
in item (2) above.
If $W_\omega$ is a wall of $C_\omega$, then we also
say that $C_\omega$ is \emph{adjacent} to $W_\omega$.

\begin{defn}\label{side}
 Let $W_\omega=\omega{\textrm -}\lim W_i$ be an $\omega$-wall. A \emph{side} $S(W_\omega)$
of $W_\omega$ is a subset $S(W_\omega)\subseteq X_\omega$ which is defined as follows.
For every $i$, let $\Omega_i$ be a connected component of $X\setminus W_i$.
Then
$$
S(W_\omega)=\left(\omega{\textrm -}\lim \Omega_i\right)\setminus W_\omega\ .
$$
\end{defn}

An internal wall has exactly
two sides $S(W_\omega)$, $S'(W_\omega)$
and is adjacent to exactly two $\omega$-chambers $C_\omega$, $C'_\omega$. Up to reordering them,
we also have $S(W_\omega)\cap C'_\omega=S'(W_\omega)\cap C_\omega=\emptyset$,
$C_\omega\setminus W_\omega\subseteq S(W_\omega)$ and $C'_\omega\setminus W_\omega\subseteq S'(W_\omega)$.
Moreover, 
any Lipschitz path joining points contained in distinct
sides of $W_\omega$ must pass through $W_\omega$.

\begin{defn}\label{esssep:def}
 A subset $A\subseteq X_\omega$ is \emph{essentially separated} by an internal $\omega$-wall $W_\omega\subseteq X_\omega$
 if $A$ intersects both sides of $W_\omega$.
\end{defn}

\begin{prop}\label{utile}
 Let $A\subseteq X_\omega$ be a subset which is not essentially separated by any $\omega$-wall. Then $A$ is contained in an $\omega$-chamber.
\end{prop}
\begin{proof}
 In the $3$-dimensional case, this is Lemma 3.4 in Kapovich and Leeb's paper~\cite{kapleenew}. Their proof applies verbatim also to the case of irreducible HDG manifolds (thus providing a positive
 answer to the problem posed at page 122 in~\cite{FLS}).
\end{proof}

As already mentioned in the introduction, a key ingredient for the analysis of the geometry of
$X_\omega$ (or of $X$) is the understanding of which subspaces separate (or coarsely separate) some relevant subsets
of $X_\omega$ (or of $X$). 
Let $S(W_\omega)$ be a side of $W_\omega$, and let $C_\omega$
be the unique $\omega$-chamber of $X_\omega$ which intersects $S(W_\omega)$
and is adjacent to $W_\omega$. 
 A fiber of $W_\omega$ associated
to $S(W_\omega)$ is a fiber of $C_\omega$
that is contained in $W_\omega$. 
The following observation  follows from the fact that the gluings defining our manifold $M$ are transverse, and it is crucial to our purposes:

\begin{lemma}\label{fiber:intersection:lem}
Let $S^+(W_\omega)$ and $S^-(W_\omega)$ be the sides of 
the internal $\omega$-wall $W_\omega$,
and let $F^+_\omega$, $F^-_\omega$ be  fibers of $W_\omega$ associated respectively to
$S^+(W_\omega)$, $S^-(W_\omega)$. Then $|F^+_\omega\cap F^-_\omega|\leq 1$.
\end{lemma}

If $P,P'$ are distinct pieces of a tree-graded space $Y$, then there exist $p\in P$, $p'\in P'$
such that,
for any continuous path $\gamma\colon [0,1]\to Y$ with $\gamma (0)\in P$ and $\gamma (1)\in P'$,
we have $p, p'\in {\rm Im}\, \gamma$ (see e.g.~\cite[Lemma 8.8]{FLS}). Since $\omega$-chambers are 
(bi-Lipschitz homeomorphic to) products
of tree-graded spaces with Euclidean factors, this immediately implies that,
if $W_\omega$ and $W'_\omega$ are distinct $\omega$-walls of the $\omega$-chamber
$C_\omega$, then there exists an $\omega$-fiber $F_\omega\subseteq W_\omega$ 
%and
%$F'_\omega\subseteq W'_\omega$ 
of $C_\omega$ such that
every continuous path in $C_\omega$ joining a point in $W_\omega$ with a point
in $W'_\omega$ has to pass through $F_\omega$.  With some work it is possible to extend this result
to pairs of $\omega$-walls
which are not contained in the same $\omega$-chamber:

\begin{lemma}[Lemma 8.24 in \cite{FLS}]\label{inters:lem}
Let $W_\omega,W'_\omega$ be distinct $\omega$-walls, and
let
$S(W_\omega)$ be the side of $W_\omega$ containing $W'_\omega\setminus W_\omega$.
Then
there exists an $\omega$-fiber $F_\omega$ of $W_\omega$ such that
\begin{enumerate}
 \item $F_\omega$ is associated to $S(W_\omega)$, and
 \item every Lipschitz path joining a point in $W'_\omega$ with
a point in $W_\omega$ passes through $F_\omega$. 
\end{enumerate}
\end{lemma}

As every point in $X_\omega$ is contained in an $\omega$-wall, Lemma~\ref{inters:lem} implies the following:

\begin{cor}\label{inters:cor}
Let $W_\omega$ be an $\omega$-wall, 
let $p\in X_\omega\setminus W_\omega$, and let 
$S(W_\omega)$ be the side of $W_\omega$ containing 
$p$.
Then
there exists an $\omega$-fiber $F_\omega$ of $W_\omega$ associated to $S(W_\omega)$
such that 
every Lipschitz path joining $p$ with
$W_\omega$ passes through $F_\omega$. 
\end{cor}

\subsection{A characterization of bi-Lipschitz flats in higher dimension}
Throughout this subsection we assume that $n=\dim M\geq 4$.
As already mentioned in Subsection~\ref{versus:sub}, in this case  the fact that fibers have higher codimension
allows us to provide an easy characterization of $\omega$-walls. 

A \emph{bi-Lipschitz $m$-flat} in $X_\omega$ is the image of a 
bi-Lipschitz embedding $f\colon \mathbb{R}^m\to X_\omega$.
This section is aimed at proving that $\omega$-walls are the only
bi-Lipschitz $(n-1)$-flats in $X_\omega$. 
We say that a metric space is L.-p.-connected if
any two points in it may be joined by a Lipschitz path. The following lemma 
provides a fundamental step towards the desired characterization of $\omega$-walls, so we give a complete
proof of it. It
breaks down in the $3$-dimensional case. 

\begin{lemma}\label{fundamental:lem}
Let $A\subseteq X_\omega$ be a bi-Lipschitz $(n-1)$-flat. Then for every $\omega$-fiber $F_\omega$
the set $A\setminus F_\omega$ is L.-p.-connected.
\end{lemma}
\begin{proof}
Let $f\colon \mathbb{R}^{n-1}\to C_\omega$ be a bi-Lipschitz embedding
such that $f(\mathbb{R}^{n-1})=A$, 
and let $l\leq n-3$ be such that $F_\omega$ is bi-Lipschitz homeomorphic
to $\mathbb{R}^l$.
The set $f^{-1}(F_\omega)$ is a closed subset
of $\mathbb{R}^{n-1}$ which is bi-Lipschitz homeomorphic to a subset
of $\mathbb{R}^l$. But it is known that the complements
of two homeomorphic closed subsets of $\mathbb{R}^{n-1}$ have the same singular homology
(see e.g.~\cite{dold}), so $\mathbb{R}^{n-1}\setminus f^{-1}(F_\omega)$ is
path-connected. It is immediate to check that any two points in a connected
open subset of $\mathbb{R}^{n-1}$ are joined by a piecewise linear path,
so $\mathbb{R}^{n-1}\setminus f^{-1}(F_\omega)$ is L.-p.-connected.
The conclusion follows from the fact that $f$ takes Lipschitz paths
into Lipschitz paths.
\end{proof}

Let now $A\subseteq X_\omega$ be a bi-Lipschitz $(n-1)$-flat. We first observe that $A$ is not essentially separated by any $\omega$-wall of $X_\omega$.
In fact, if this were not the case, then there would exist 
an $\omega$-wall $W_\omega$ and points  $p,q\in A$ on opposite sides of $W_\omega$.
Then 
the fiber $F_\omega$ of $W_\omega$ such that every path joining $p$ with
$W_\omega$ passes through $F_\omega$ (see Corollary~\ref{inters:cor}) would disconnect $A$, against
Lemma~\ref{fundamental:lem}.
By Proposition~\ref{utile} we can then suppose that $A$ is contained in an $\omega$-chamber
$C_\omega$. 
 Recall that $C_\omega$ is homeomorphic to a product $Y\times \mathbb{R}^l$,
where $Y$ is a tree-graded space and $l\leq n-3$. 
Lemma~\ref{fundamental:lem} implies that the projection of $A$ onto $Y$ does not have cut-points, so it is contained
in a piece of $Y$. This is equivalent to say that $A$ is contained in an $\omega$-wall of $C_\omega$, so it is actually
equal to such an $\omega$-wall by invariance of domain. 
We have thus proved the following:

\begin{prop}\label{wall:char:prop}
Let $A$ be a bi-Lipschitz $(n-1)$-flat in $X_\omega$.
Then
$A$ is an $\omega$-wall.
\end{prop}

\begin{cor}\label{omegawalls:preserved}
 Let $M,M'$ be irreducible HDG manifolds with universal coverings $X,X'$, and let $f\colon X\to X'$ be a quasi-isometry.
 Then the map $f_\omega\colon X_\omega\to X'_\omega$ induced by $f$ on the asymptotic cones takes any $\omega$-wall
 of $X_\omega$ onto an $\omega$-wall of $X'_\omega$.
\end{cor}

\subsection{The $3$-dimensional case}
We would like to extend Corollary~\ref{omegawalls:preserved} to the $3$-dimensional case.
Let $M,M'$ be closed irreducible non-geometric $3$-manifolds with universal coverings $X,X'$. We can suppose that $M,M'$ are non-positively curved. We will say that an $\omega$-chamber
of the asymptotic cone $X_\omega$ (or $X'_\omega$) is hyperbolic (resp. Seifert) if it is the $\omega$-limit of chambers covering a hyperbolic (resp. Seifert) piece
of $M$ (or $M'$). Observe that every $\omega$-chamber is either Seifert or hyperbolic. Being bi-Lipschitz homeomorphic to the product of a real tree with the real line, every Seifert $\omega$-chamber
contains many flats that are not $\omega$-walls, so Proposition~\ref{wall:char:prop} cannot hold in this case. 
Therefore, in order to obtain Corollary~\ref{omegawalls:preserved}
some additional arguments are needed, that we briefly outline here.
The reader is addressed to~\cite{kapleenew}
for complete proofs.

We first collect some facts about (bi-Lipschitz) flats in $X_\omega$.

\begin{enumerate}
\item Any bi-Lipschitz flat contained in an $\omega$-chamber is a flat (this easily follows from the explicit description of the geometry of the chambers,
together with the basic properties of tree-graded spaces and the fact that, thanks to non-positive curvature, $\omega$-chambers are isometrically embedded in $X_\omega$).
 \item Any flat $Z_\omega\subseteq X_\omega$ must be contained in a single $\omega$-chamber. In fact, otherwise $Z_\omega$ is essentially separated 
 by an $\omega$-wall $W_\omega$. The arguments described in the previous subsection may be exploited to show that no $\omega$-chamber adjacent
to $W_\omega$ can be hyperbolic, i.e.~the two $\omega$-chambers adjacent to $W_\omega$ must both be Seifert. Moreover, $Z_\omega\cap W_\omega$ must contain two transversal fibers of
$W_\omega$, and this implies that indeed $Z_\omega=W_\omega$, a contradiction.
\item A bi-Lipschitz flat $B\subseteq X_\omega$ is an $\omega$-wall which is not adjacent to any Seifert $\omega$-chamber if and only
of its intersection with any other bi-Lipschitz flat in $X_\omega$ contains at most one point. In fact, if the latter condition is true, then $B$ cannot be essentially separated by any $\omega$-wall,
so it is contained in an $\omega$-chamber that cannot be Seifert (because otherwise $B$ would intersect many other flats in more than one point). Conversely, if $B$ is an $\omega$-wall
not adjacent to any Seifert $\omega$-chamber and $B'$ is a bi-Lipschitz flat intersecting $F$, then the closure of any component of
$B'\setminus B$ intersects $B$ in one point, which of course cannot disconnect $B$. Therefore $B'\setminus B$ is connected and $B'\cap B$ is one point.
\item If $T$ is a real tree which branches at every point, then the image $\Omega$ of any bi-Lipschitz embedding $f\colon T\times \mathbb{R}\to X_\omega$ is contained
in a Seifert $\omega$-chamber. In fact, if $W_\omega$ is an $\omega$-wall such that $S^\pm(W_\omega)\cap \Omega=\Omega^\pm$ are both non-empty, then
the boundaries $l^\pm$ of $\Omega^\pm$ are transversal fibers in $W_\omega$. However, their inverse image under $f$ separate $T\times\mathbb{R}$,
so they should be parallel lines in $T\times\mathbb{R}$, and this contradicts the fact that $f$ is bi-Lipschitz. The conclusion now follows from Proposition~\ref{utile}.
\end{enumerate}

These facts imply the following:

\begin{prop}\label{primifatti}
 Let $f\colon X_\omega\to X'_\omega$ be a bi-Lipschitz homeomorphism.
 Then:
 \begin{enumerate}
  \item[(i)]
  $f$ maps flats to flats;
  \item[(ii)]
 $f$ maps each $\omega$-wall that is not adjacent to a Seifert $\omega$-chamber into an $\omega$-wall of the same kind;
 \item[(iii)]
 $f$ maps any Seifert $\omega$-chamber of $X_\omega$ into a Seifert $\omega$-chamber of $X'_\omega$.
 \end{enumerate}
\end{prop}
\begin{proof}
 By fact (2) above, any flat in $X_\omega$ is contained in an $\omega$-chamber. Thus assertion (i) follows from fact (3) for $\omega$-walls
 adjacent to at least one hyperbolic $\omega$-chamber and from fact (4) for flats contained in Seifert $\omega$-chambers. Assertions (ii) and (iii)
 follow from
 facts (3) and (4), respectively.
\end{proof}

We are now ready to prove the following result, which shows that also in the $3$-dimensional case $\omega$-walls and $\omega$-chambers
are preserved by bi-Lipschitz homeomorphisms:

\begin{cor}\label{omegawalls3dim}
 Let $f\colon X_\omega\to X'_\omega$ be a bi-Lipschitz homeomorphism. Then:
 \begin{enumerate}
  \item $f$ maps any $\omega$-chamber of $X_\omega$ into an $\omega$-chamber of the same type;
  \item $f$ maps $\omega$-walls into $\omega$-walls. 
  \end{enumerate}
\end{cor}
\begin{proof}
Since every $\omega$-wall may be expressed as the intersection of adjacent $\omega$-chambers,
(2) readily follows from (1). In order to show (1), since we already know that Seifert $\omega$-chambers are sent to Seifert $\omega$-chambers (see Proposition~\ref{primifatti}--(iii)), it is sufficient
to prove that any hyperbolic $\omega$-chamber $C_\omega\subseteq X_\omega$ is sent to an $\omega$-chamber of $X'_\omega$ (that cannot be Seifert by Proposition~\ref{primifatti}--(iii) applied to $f^{-1}$). 
Suppose by contradiction that $f(C_\omega)$ is essentially separated by the $\omega$-wall $W'_\omega\subseteq X'_\omega$. We know that $C_\omega$ is the union of its $\omega$-walls, that every $\omega$-wall lies in an
$\omega$-chamber, and that $f$ takes flats to flats (see Proposition~\ref{primifatti}--(i)), so there exist flats $Z^+,Z^-\in f(C_\omega)$ such that $Z^\pm \cap S^\pm (W'_\omega)\neq\emptyset$. Now the bi-Lipschitz flats
$f^{-1}(Z^\pm)$ lie in the hyperbolic $\omega$-chamber $C_\omega$, so they are flats, and are essentially separated by the flat $f^{-1}(W'_\omega)\subseteq X_\omega$. This provides the desired contradiction, because
two flats in the same hyperbolic $\omega$-chamber cannot be essentially separated by any other flat.
\end{proof}

\section{Proof of Theorem~\ref{main1:thm}}\label{main1:sec}
Let now $M,M'$ both be either closed irreducible non-geometric $3$-manifolds, or  irreducible $n$-dimensional HDG manifolds, $n\geq 3$. 
We denote by $X,X'$ the Riemannian universal coverings of $M,M'$, respectively. In the case when $\dim M=\dim M'=3$, we can (and we do) suppose that $X,X'$ are endowed with
a non-positively curved equivariant metric.  We also fix a quasi-isometry $f\colon X\to X'$. We have already observed that 
Theorems~\ref{kaplee} and~\ref{qi-preserve:thm} reduce to Theorem~\ref{main1:thm}, which asserts that:
\begin{enumerate}
 \item 
for every chamber $C$ of $X$ there exists a chamber $C'$ of $X'$ such that
  $f(C)$ lies within finite Hausdorff distance from $C'$;
  \item
  $f$ preserves the structures of $X,X'$ as tree of spaces.
\end{enumerate} 

Our study of bi-Lipschitz homeomorphisms between the asymptotic cones $X_\omega$ and $X'_\omega$ was mainly aimed at showing that
such homeomorphisms must preserve $\omega$-walls (see Corollaries~\ref{omegawalls:preserved} and~\ref{omegawalls3dim}). We first show how this fact may be used
to show that $f$ must coarsely preserve walls. The proof of the following result illustrates a very general strategy to get back from homeomorphisms between
asymptotic cones to quasi-isometries between the original spaces.

\begin{thm}\label{wall:thm}
There exists $\beta\geq 0$ 
such that, if $W$ is a wall of $X$, then $f(W)$  lies at Hausdorff distance bounded by $\beta$ from a wall $W'\subseteq X'$.
Moreover, $f$ stays at distance bounded by $\beta$ from a quasi-isometry between $W$ and $W'$.
\end{thm}

\begin{proof}
It is well-known that a quasi-isometric embedding between spaces both quasi-isometric to $\mathbb{R}^{n-1}$ is itself a quasi-isometry
(for example, because otherwise by taking asymptotic cones with suitably chosen sequences of basepoints and of rescaling factors,
one could construct a non-surjective bi-Lipschitz embedding of $\mathbb{R}^{n-1}$ into itself, which cannot exist). 
Therefore, it is sufficient to show that there exists a wall $W'\subseteq X'$ such that $f(W)\subseteq N_\beta(W')$, where
$N_\beta$ denotes the $\beta$-neighbourhood in $X'$.

Suppose by contradiction that 
for each $m\in\mathbb{N}$ and wall $W'\subseteq X'$ we have $f (W)\nsubseteq N_m (W')$.
Fix a point $p\in W$. The quasi-isometry $f$ induces a bi-Lipschitz homeomorphism
$f_\omega$ from the asymptotic cone $X_\omega =X_\omega ((p),(m))$ 
to the asymptotic cone $X'_\omega (f (p),(m))$. 
By Corollaries~\ref{omegawalls:preserved} and~\ref{omegawalls3dim},
if $W_\omega$ is the $\omega$-limit of the constant sequence of subsets all equal to $W$, then
there is an $\omega-$wall $W'_\omega=\omega$-$\lim W'_m$ such that 
$f_\omega(W_\omega)=W'_\omega$. 
By hypothesis, for each $m$ there is a point $p_m\in W$ with $d(f (p_m),W'_m)\geq m$.
Set $r_m = d(p_m,p)$. By choosing $p_m$ as close to $p$ as possible, 
we may assume that no point $q$ such that $d(p,q)\leq r_m-1$ satisfies $d(f(q),W'_m)\geq m$, 
so 
\begin{equation}\label{cond1:eqn}
d(f (q),W'_m)\leq m +k+c\qquad  {\rm for\ every}\ q\in W\ {\rm s.t.}\ d(p,q)\leq r_m.
\end{equation} 
Notice that $\omega$-$\lim r_m/m=\infty$,
for otherwise $[(p_m)]$ should belong to $W_\omega$,
$[f (p_m)]$ should belong to $X'_\omega ((f(p)), (m))$, and, since
$f_\omega(W_\omega)=W'_\omega$, we would
have $d(f(p_m),W'_m)=o(m)$. 

Let us now change basepoints, and consider instead the pair of asymptotic cones
$\overline{W}_\omega=W_\omega( (p_m), (m))$ and $X'_\omega ((f(p_m)), (m) )$.
The quasi-isometry $f$ induces a bi-Lipschitz embedding $\overline f$ between these asymptotic cones
(note that $f\neq \overline{f}$, simply because due to the change of basepoints, 
$f$ and $\overline{f}$ are defined on different spaces
with values in different spaces!).
Let $A_m=\{q\in W\, |\, d(q,p)\leq r_m\}$ and $A_\omega=
\omega$-$\lim A_m\subseteq \overline{W}_\omega$.
Since $\omega$-$\lim r_m/m=\infty$, it is easy to see that $A_\omega$ is bi-Lipschitz
homeomorphic to a half-space in $\overline{W}_\omega$. Moreover, 
by (\ref{cond1:eqn}) each point in $\overline{f}(A_\omega)$ is at a distance at most 1 from 
$\overline{W}'_\omega=\omega$-$\lim W'_i\subseteq X'_\omega ((f(p_m)), (m))$ (as before, observe that the sets $\overline{W}'_\omega$ and $W'_\omega$
live in different spaces). 
Again by Corollaries~\ref{omegawalls:preserved} and~\ref{omegawalls3dim},
we have that $\overline{f}(A_\omega)\subseteq \overline{f}(\overline{W}_\omega )= W''_\omega$ 
for some $\omega-$wall $W''_\omega$. Moreover, since $[(f (p_m))]\in W''_\omega\setminus
\overline{W}'_\omega$, we have $\overline{W}'_\omega \neq W''_\omega$. 

By Lemma~\ref{inters:lem} there exists a fiber $F_\omega\subseteq \overline{W}'_\omega\cap W''_\omega$ such that
every path joining a point in $W''_\omega$
with a point in $\overline{W}'_\omega$ has to pass through $F_\omega$. Now, if $a\in \overline{f}(A_\omega)$  
we have $d(a, \overline{W}'_\omega)\leq 1$, so there exists a geodesic of length at most one 
joining $a\in W''_\omega$ with some point in $\overline{W}'_\omega$. Such a geodesic must pass
through $F_\omega$, so 
every point of $\overline{f}(A_\omega)$ must be at a distance at most 1 from $F_\omega$.
If $h\colon \overline{f}(A_\omega)\to F_\omega$ is such that $d(b,h(b))\leq 1$ for every $b\in \overline{f}(A_\omega)$,
then $h$ is a $(1,2)$-quasi-isometric embedding. Therefore the map
$g=h\circ \overline{f}\colon A_\omega\to F_\omega$ is a quasi-isometric embedding. But this is not possible,
since if $n-1>l$ there are 
no quasi-isometric embeddings from a half space in $\mathbb{R}^{n-1}$ to $\mathbb{R}^l$ (as, taking asymptotic cones, 
such an embedding would provide an injective continuous function 
from an open set in $\mathbb{R}^{n-1}$ to $\mathbb{R}^l$). This completes the proof of the theorem.
\end{proof}

If $W\subseteq X$ is a wall, then we denote by $f_\#(W)$ the wall of $X'$ which lies at finite Hausdorff distance
from $f(W)$ (this wall exists by the previous theorem, and it is obviously unique because distinct walls in $X'$ are at infinite
Hausdorff distance from each other).
It is not difficult to show that, if $W_1,W_2,W_3$ are walls of $X$ such that
$W_3$ does not separate $W_1$ from $W_2$, then we can connect $W_1$ and $W_2$ by a curve supported outside $N_r(W_3)$, where $r>0$ is any chosen constant. 
If $r$ is sufficiently large, then the image of such a curve via $f$ may be replaced by a continuous path that connects
$f_\#(W_1)$ and $f_\#(W_2)$ without intersecting $f_\# (W_3)$. This readily implies that $f$ maps walls that are adjacent to the same chamber
(close to) walls that are adjacent to the same chamber. Together with the fact that the walls of a chamber are $r$-dense in the chamber itself for some $r>0$, 
this implies in turn that $f$ must coarsely preserve chambers. Moreover, $f$ must also preserve the structures of $X$ and of $X'$ as trees of spaces. This concludes the proof
of Theorem~\ref{main1:thm}, whence of Theorems~\ref{kaplee} and~\ref{qi-preserve:thm}.

\section{Quasi-isometric rigidity}\label{QIrigidity:sec}
Theorem 4.1 allows us to understand quite well groups that are quasi-isometric to fundamental groups of irreducible non-geometric $3$-manifolds or of irreducible HDG
manifolds. We begin by studying groups quasi-isometric to fundamental groups of the pieces. As usual, the case when Seifert pieces are allowed requires more care, so we first focus on the higher dimensional case. 

\subsection{QI-rigidity of higher dimensional pieces}\label{subseccita}
Let $N$ be a complete finite-volume hyperbolic $m$-manifold, $m\geq 3$, and
let $\Gamma$ be a finitely generated group quasi-isometric to
$\pi_1 (N)\times\mathbb{Z}^d$, $d\geq 0$. Without loss of generality, we may assume that the cusps of $N$ are toric.

By definition, the group $\G$ is quasi-isometric to $\pi_1(M)$, where
$M=\overline{N}\times T^d$ is an (obviously irreducible) HDG manifold of dimension $n=m+d$ consisting of a single piece.
The universal covering $X$ of $M$ is isometric to the Riemannian
product $B\times \mathbb{R}^d$, where $B$ is a neutered space. The walls of $X$ coincide
with the boundary components of $X$.
As discussed  
in Section~\ref{definitions:sec}, a quasi-isometry between $\Gamma$ and $\pi_1 (M)$ induces a geometric 
quasi-action of $\Gamma$ on $X$ that will be fixed from now on. As usual, we will identify every element $\gamma\in\G$ with the corresponding quasi-isometry
of $X$ defined by the quasi-action.

We want to prove that every quasi-isometry $\gamma\colon X\to X$, $\gamma\in \Gamma$ 
can be coarsely projected on $B$ to obtain a quasi-isometry of $B$. 
Recall first that Theorem~\ref{wall:thm} implies that every wall of $X$ is taken by $\gamma$ close to another wall. Since every fiber of $X$
may be expressed as the coarse intersection of two walls, this readily implies that $\gamma$ must coarsely preserve also fibers.
This fact can be exploited  to define a quasi-action of $\Gamma$ on $B$ as follows: 
for every $\gamma\in \Gamma$,
we define a map $\psi (\gamma)\colon B\to B$ by setting 
$\psi(\gamma)(b)=\pi_B (\gamma( (b,0) ))$ for every $b\in B$, where $\pi_B\colon X\cong B\times \mathbb{R}^d \to B$
is the natural projection. It is not difficult that $\psi$ is indeed a cobounded quasi-action.
On the contrary, the fact that the quasi-action $\psi$ is proper is  a bit more delicate.

We first observe that, from the way the action of $\Gamma$ on
$B$ was defined, every $\gamma\in\Gamma$ coarsely permutes the components of $\partial B$.
Recall that $m=n-d$ is the dimension of the neutered space $B$, and let $G$
be the isometry group of $(B,d_B)$. Every element of $G$ is the restriction to $B$ of an isometry
of the whole hyperbolic space $\mathbb{R}^m$ containing $B$.
We will denote by ${\rm Comm} (G)$ the \emph{commensurator} of 
$G$ in ${\rm Isom} (\mathbb{R}^m)$, i.e.~the group
of those elements $h\in {\rm Isom} (\mathbb{R}^m)$ such that
the intersection $G\cap (h G h^{-1})$
has finite index both in $G$ and in $h G h^{-1}$.

The following rigidity result is an important step in Schwartz's proof of QI-rigidity of non-uniform lattices in $G$:

\begin{prop}[Lemma 6.1 in \cite{Schwartz}]\label{schwartz1}
For every $\gamma\in\Gamma$ there exists a unique isometry
$\theta (\gamma)\in {\rm Isom}(\mathbb{R}^m)$ 
whose restriction to $B$ stays at finite distance from $\psi(\gamma)$.
Moreover, for every $\gamma\in\Gamma$ the isometry
$\theta (\gamma)$ belongs to ${\rm Comm} (G)$, and 
the resulting map $\theta\colon \Gamma\to {\rm Comm}(G)$ is
a group homomorphism.
\end{prop} 

In order to conclude our study of $\G$ we now need to understand the structure of the kernel and of the image of $\theta$.
We set $\Lambda=\theta(\G)< {\rm Isom}(\mathbb{R}^m)$ the image of the
homomorphism $\theta$, and we show that $\Lambda$ is 
commensurable with $\pi_1 (N)$. 
It is a result of Margulis that a non-uniform lattice in ${\rm Isom} (\mathbb{R}^m)$ is arithmetic
if and only if it has infinite index in its commensurator (see~\cite{zim}). As a result,
things would be quite a bit easier if $N$ were assumed to be non-arithmetic. 
To deal with the general case, one needs to observe the following facts:
\begin{enumerate}
\item Since elements of $\Lambda$ are uniformly close to quasi-isometries of $B$, each of them must send every horosphere
in $\partial B$ into a horosphere $O'$ which is parallel and uniformly close to a horosphere $O''\subseteq B$;
\item using (1), one can slightly change the ``heights'' of the horospheres in $\partial B$ in order to
define a new neutered space $\widehat{B}$ which is left invariant by the action of $\Lambda$;
\item since the isometry group of $\widehat{B}$ is discrete, one then gets that $\Lambda$ is discrete; being cobounded
on $\widehat{B}$, the action of $\Lambda$ on $\mathbb{H}^m$ has a finite covolume, so $\Lambda$ is a non-uniform lattice;
\item being quasi-isometric to $\widehat{B}$ and $B$ respectively, the groups $\Lambda$ and $\pi_1(N)$ are quasi-isometric to each other,
so one may use again Schwartz's results to conclude that $\Lambda$ is commensurable with $\pi_1(N)$.
\end{enumerate}

The study of $\ker \theta$ is easier. In fact, it is not difficult to show that the quasi-action of $\ker\theta$ on $X$
may be slightly perturbed to define a geometric  quasi-action of $\ker\theta$ on one fiber of $X$. As a consequence, $\ker\theta$ is finitely generated,
quasi-isometric to $\mathbb{Z}^d$ and quasi-isometrically embedded in $\G$. Since groups quasi-isometric to $\mathbb{Z}^d$ are virtually isomorphic to $\mathbb{Z}^d$,
we have thus shown that $\G$ is isomorphic to the extension of a non-uniform lattice commensurable with $\pi_1(N)$ by a group virtually isomorphic to $\mathbb{Z}^d$.
Since abelian undistorted normal subgroups are always virtually central (see~\cite[Proposition 9.10]{FLS}), this concludes the proof of Theorem~\ref{product:thm}.

%%%%%%%%%%%%%%%%%%

\subsection{QI-rigidity of $3$-dimensional pieces}
Something more can be said in the $3$-dimensional case. Namely, if $\G$ is quasi-isometric to the fundamental
group of a hyperbolic piece, then Schwartz's results imply that $\G$ is a finite extension of a non-uniform lattice
in ${\rm Isom}(\mathbb{H}^3)$, i.e.~$\G$ fits into a short exact sequence
$$
\xymatrix{
1 \ar[r] &F\ar[r]& \G\ar[r] &\G'\ar[r]& 1\ ,
}
$$
where $F$ is finite and $\G'$ is the fundamental group of a finite-volume hyperbolic $3$-orbifold with flat cusps. Moreover, 
(quasi-)stabilizers in $\G$ of boundary flats of $X$ are sent
to peripheral subgroups of $\G'$, and the subgroup
$F$ can be characterized
as the maximal finite normal subgroup of $\G'$, and is also the unique maximal finite normal subgroup
of every (quasi-)stabilizer of boundary flats of $X$.

In the Seifert case we have $X=B\times\mathbb{R}$, where $B$ may be chosen to be the complement in $\mathbb{H}^2$ of an equivariant
collection of disjoint open half-planes (horoballs are replaced by half-planes because the base surface is now geometrized as a
surface with geodesic boundary rather than with cusps). Schwart's results are no longer available in dimension 2, but using the fact that all the (quasi-)isometries
involved (quasi)-preserve the boundary of $B$ it is possible to slighlty modify the strategy described above to show that $\G$
fits into a short exact sequence
$$
\xymatrix{
1\ar[r] & K\ar[r] & \G \ar[r] & \G'\ar[r] & 1\ ,
}
$$
where $K$ has a unique maximal finite normal subgroup $F$, the group
$K/F$ is isomorphic either to $\mathbb{Z}$ or to the infinite dihedral group, and $\G'$ is the fundamental
group of a compact hyperbolic $2$-orbifold with geodesic boundary. Again the {(quasi-)stabilizer} of any boundary component
of $X$ is sent to a peripheral subgroup of $\pi_1(O)$, and $F$ is also the unique maximal normal subgroup of the (quasi-)stabilizers
of the boundary components of $X$.

\subsection{Quasi-isometric rigidity: the final step}
Let now $M$ be either an irreducible non-geometric $3$-manifold or an irreducible HDG manifold with universal covering $X$,
and take a group $\G$ quasi-isometric to $\pi_1(M)$. By Proposition~\ref{quasiact:prop} we have a geometric quasi-action of $\G$ on $X$.
By Theorem~\ref{main1:thm}, this quasi-action induces an action by automorphisms on the simplicial tree $T$ which encodes the structure of $X$
as a tree of spaces. Now a fundamental result from Bass-Serre theory says that any group acting on a simplicial tree without inversions is isomorphic
to the fundamental group of a graph of groups whose vertex groups coincide with (conjugates of) the stabilizers of vertices, and edge groups
coincide with (conjugates of) the stabilzers of edges (recall that $G$ acts on  $T$ \emph{without inversions} if no
element of $G$ switches the endpoints of an edge of $T$).

Now the action of $\Gamma$ on $T$  might include
some inversions, but it is easy to construct a subgroup
$\G^0$ of $\G$ of index at most two that acts on $T$ without inversions.
Moreover, it readily follows from the construction that vertex groups of $\G^0$ are quasi-isometric to stabilizers of chambers, while edge groups
are quasi-isometric to stabilizers of walls. This already concludes the proof of quasi-isometric rigidity in the higher dimensional case
(i.e.~Theorem~\ref{qirigidity:thm}).

In the $3$-dimensional case a stronger result holds, thanks to the extra information we described above. In fact,
after replacing $\G$ with $\G^0$ (which is obviously virtually isomorphic to $\G$), we can make use of the fact
that the unique maximal finite normal subgroups of all vertex and edge stabilizers coincide, and therefore coincide with
the kernel $F$ of the action of $\G$ on $T$. The vertex stabilizers for the action
of $\G'=\G/F$ on $T$ are fundamental groups of hyperbolic or Seifert $3$-orbifolds with boundary. One can then glue these orbifolds together
(according to the combinatorics described by the graph $T/\G$) to get a $3$-dimensional orbifold with fundamental group $\G'$.
This $3$-orbifold is finitely covered by a manifold~\cite{MM}, and this concludes the proof of Theorem~\ref{KL2}. 

\section{Open questions}\label{open}
As already stated in the introduction, 
in order to conclude the
  classification of $3$-manifold groups up to quasi-isometry only the case of non-geometric irreducible manifolds with at least one arithmetic hyperbolic piece
 still has to be understood (see~\cite{BN2}). Therefore, in this section we pose some questions about the quasi-isometric rigidity of higher dimensional piecewise geometric manifolds.
 Many of the following problems are taken from~\cite[Chapter 12]{FLS}.

 The following question in addressed in~\cite{FK}:
 
 \begin{problem}
 In~\cite{Nguyen1,Nguyen2}
 Nguyen Phan  defined the class of \emph{cusp decomposable manifolds}, each element of which decomposes into cusped locally symmetric pieces. Is is true than any quasi-isometry
  between the fundamental groups of two cusp decomposable manifolds induces a quasi-isometry between the fundamental groups of their pieces? To what extent does quasi-isometric rigidity
  hold for fundamental groups of cusp decomposable manifolds?
 \end{problem}

 \begin{problem}
 In~\cite{LS}, Leeb and Scott
  defined a canonical decomposition for non-positively curved closed Riemannian manifolds, which provides a generalization to higher dimensions
  of the JSJ decomposition of irreducible $3$-manifolds. Is it true than every quasi-isometry between the fundamental groups of two closed non-positively curved Riemannian manifolds induces a 
  quasi-isometry between the fundamental groups of their pieces? 
 \end{problem}

 Specializing
to the class of HDG manifold groups, Theorem \ref{qirigidity:thm} describes a necessary condition for deciding
whether the fundamental groups of two irrfeducible HDG manifolds $M_1$ and $M_2$ are quasi-isometric to each other: the two HDG manifolds $M_i$ must essentially be built up from the same collection of 
pieces (up to commensurability), with the same patterns of gluings (see~\cite[Theorem 10.7]{FLS} for a precise statement). The only
distinguishing feature between $M_1$ and $M_2$ would then be in the actual gluing maps
used to attach pieces together. We are thus lead to the following questions:

\begin{problem}\label{glueing-QI}
Take pieces $V_1$ and $V_2$ each having exactly one 
boundary component, and let $M_1,M_2$ be a pair of irreducible HDG manifolds obtained 
by gluing $V_1$ with $V_2$. Must the fundamental groups of $M_1$ and $M_2$ be 
quasi-isometric? 
\end{problem}

\begin{problem}\label{CATpro}
Is there a pair of irreducible HDG manifolds with quasi-isometric fundamental groups, with the
property that one of them supports a locally CAT(0) metric, but the other
one cannot support any locally CAT(0) metric? 
\end{problem}

\begin{problem}\label{glueing-NP}
 Is it true that the fundamental group of every irreducible HDG manifold is quasi-isometric to the fundamental group
 of a non-positively curved HDG manifold? 
\end{problem}

\begin{problem}\label{glueing-AB}
 Is it true that the fundamental group of every irreducible HDG manifold is semihyperbolic in the sense of Alonso and Bridson~\cite{alo}?
\end{problem}

A positive answer to Problem~\ref{glueing-QI} would imply positive answers both to Problem~\ref{CATpro} and to Problem~\ref{glueing-NP}, 
and a positive answer to Problem~\ref{glueing-NP}
would imply in turn a positive answer
to Problem~\ref{glueing-AB}.

Concerning Problem~\ref{CATpro}, 
Nicol's thesis~\cite{Nicol} exhibits in each dimension $\geq 4$ infinitely many
pairs of \emph{non-irreducible} HDG manifolds with quasi-isometric fundamental groups, with the
property that each pair consists of one manifold that supports a locally CAT(0) metric and one manifold that
cannot support any locally CAT(0) metric.

Notice that in the proof of Theorem~\ref{qirigidity:thm} each vertex stabilizer is studied separately.
It might be possible to obtain additional information by studying
the interaction between vertex stabilizers of adjacent vertices, just as Kapovich and Leeb did in the $3$-dimensional case~\cite{kapleenew}:

\begin{problem}
 Is it possible to strenghten the conclusion of Theorem~\ref{qirigidity:thm}?
\end{problem}

\appendix
\section{Quasi-isometric invariance of the prime decomposition}\label{prime:sec}
One of the most influential
results in geometric group theory is Stallings' Theorem, which asserts that a group has more than one end if and only if it splits non-trivially as a free product or an HNN-extension with amalgamation  over a finite subgroup
(recall that the number of ends of any proper geodesic metric space is a quasi-isometry invariant (see e.g.~\cite[Proposition 8.29]{BH}), so
it makes sense to speak of the number of ends of a group).
Before going into the details of the proofs of Theorems~\ref{prime:thm} and~\ref{prime2:thm}, let us recall
some terminology and the main results from~\cite{PapaWhyte}. We say that a graph of groups is \emph{terminal} if every edge group is finite and every vertex group of $\Delta$ cannot be expressed as a non-trivial free product or HNN-extension
amalgamated along finite subgroups. By Stallings' Theorem, this is equivalent to say that every vertex group has less than two ends. 
We also say that a terminal graph of groups is a terminal decomposition of its fundamental group.
It is a striking fact that not all finitely generated groups admit a terminal decomposition: in other words, 
there may be cases when one can go on splitting a finitely generated group along finite subgroups an infinite number of times~\cite{Dun2}. Groups admitting a realization
as the fundamental group of a terminal graph of groups are called \emph{accessible}. In the torsion-free case, free products amalgamated along finite subgroups are just free products, so
accessibility is guaranteed by Grushko Theorem, which asserts that the minimal number of generators of a free product is the sum of the minimal number of generators of the factors.
Moreover, it is a deep result of Dunwoody~\cite{Dun} that  every finitely presented
group is also accessible, so fundamental groups of closed $3$-manifolds are accessible (this may also be easily deduced from the existence of a prime decomposition, together with 
Grushko Theorem and the fact that aspherical manifolds have torsion free fundamental groups).

 For any group $\G$, let us denote by $e(\G)$ the number of ends of $\G$. Theorem~0.4 in~\cite{PapaWhyte} states the following:

\begin{thm}[\cite{PapaWhyte}]\label{PW:thm}
 Let $\G$ be an accessible group with terminal decomposition $\mathcal{G}$. A group $\G'$ is quasi-isometric to $\G$ if and only if the following holds: $\G'$ is also accessible, $e(\G)=e(\G')$, and any terminal decomposition
 of $\G'$ has the same set of quasi-isometry types of one-ended factors as $\mathcal{G}$.
\end{thm}

Let us now apply this result in the context of $3$-manifolds. We first prove the following:

\begin{lemma}\label{ends:lemma}
 Let $M$ be a closed $3$-manifold with fundamental group $\G$. Then:
 \begin{enumerate}
  \item $M$ is irreducible with finite fundamental group if and only if $e(\G)=0$.
  \item $M$ is irreducible with infinite fundamental group if and only if $e(\G)=1$.
  \item $M\in \{S^2\times S^1,\mathbb{P}^3(\mathbb{R})\#\mathbb{P}^3(\mathbb{R})\}$ if and only if $e(\G)=2$.
  \item $M$ is not prime and distinct from $\mathbb{P}^3(\mathbb{R})\#\mathbb{P}^3(\mathbb{R})$ if and only if $e(G)=\infty$.
 \end{enumerate}
\end{lemma}
\begin{proof}
 Since being irreducible with finite fundamental group, being irreducible with infinite fundamental group, belonging to $\{S^2\times S^1,\mathbb{P}^3(\mathbb{R})\#\mathbb{P}^3(\mathbb{R})\}$, and
 being not prime but distinct from $\mathbb{P}^3(\mathbb{R})\#\mathbb{P}^3(\mathbb{R})$ are mutually exclusive conditions, it is sufficient to prove the ``if'' implications.
 Let $M=M_1\#\ldots\# M_k$ be the prime decomposition of $M$, and set $\G_i=\pi_1(M_i)$ so that $\G=\G_1*\ldots *\G_k$. By the Poincar\'e conjecture, we have $\G_i\neq \{1\}$ for every $i=1,\ldots,k$.
 
 If $e(\G)=0$, then $\G$ is finite, so $k=1$ and $\G\neq \pi_1(S^2\times S^1)$. Therefore, $M$ has a finite fundamental group, and it is prime and distinct from $S^2\times S^1$, whence irreducible.
 
 Let now $e(\G)=1$, and suppose by contradiction that $M$ is not irreducible. Since $e(\pi_1(S^2\times S^1))\neq 2$ we have $M\neq S^2\times S^1$, so $M$ is not prime and
 $\G$ splits as a non-trivial free product $\G=\G_1*\ldots *\G_k$, $k\geq 2$. Now it is well-known that such a product has two ends if $k=2$ and $\G_1=\G_2=\mathbb{Z}_2$ and infinitely may ends otherwise
 (see e.g.~\cite[Theorem 8.32]{BH}). This contradicts the fact that $e(\G)=1$, and shows that $M$ is irreducible (the fact that $\G$ in infinite obviously follows from $e(\G)=1$).
 
 Let now $e(\G)=2$. We can argue as above to deduce that either $M$ is prime, or $k=2$ and $\G_1=\G_2=\mathbb{Z}_2$. In the second case, by the Poincar\'e conjecture both $M_1$ and $M_2$ are doubly covered by $S^3$,
 and this implies that $M_1=M_2=\mathbb{P}^3(\mathbb{R})$, since any fixed-point-free involution of the $3$-sphere is conjugated to the antipodal map~\cite{Livesay}. 
 Therefore, we may suppose that $M$ is prime, and since $S^2\times S^1$ is the only prime manifold which is not irreducible, we are left to show that $M$ is not irreducible. 
 Suppose by contradiction that $M$ is irreducible. Since groups with two ends are virtually infinite cyclic (see e.g.~\cite[Theorem 8.32]{BH}),  we can choose a finite covering $\widetilde{M}$ of $M$ with infinite cyclic fundamental group. 
 The manifold $\widetilde{M}$ is still irreducible, and
 it is well-known that irreducible $3$-manifolds
 with infinite fundamental groups are aspherical. Therefore, the cohomological dimension of $\mathbb{Z}=\pi_1(\widetilde{M})$ should be equal to three, a contradiction. This concludes the analysis of the case $e(\G)=2$.
 
 Finally suppose that $e(\G)=\infty$. Since $\pi_1(\mathbb{P}^3(\mathbb{R})\#\mathbb{P}^3(\mathbb{R}))$ has two ends,
 we need to show that $M$ is not prime. Since $\pi_1(S^2\times S^1)$ has two ends and finite groups have $0$ ends, we may suppose by contradiction that $M$ is irreducible with infinite fundamental
 group. We have already observed that this implies that $M$ is aspherical, so $\G$ is torsion free. By Stallings' Theorem, since $e(\G)>1$ the group $\G$ must split as an HNN-extension or an amalgamated product 
 over a finite subgroup. Being $\G$ torsion free, this implies that $\G$ actually splits as a non-trivial free product, so by the Kneser conjecture $M$ cannot be prime. 
\end{proof}

\begin{rem}\label{big:1ended}
Thanks to the previous lemma, a prime manifold $M$  is big
if and only if its fundamental group has one end.
 \end{rem}

Let now $M$ be a closed $3$-manifold with prime decomposition $M=M_1\# \ldots\# M_k$. The graph of groups $\mathcal{G}$ corresponding to this decomposition is not quite terminal according to the definition above, because
every summand homeomorphic to $S^2\times S^1$ gives rise to a vertex group which may be expressed as the unique HNN-extension of the trivial group. However, if in $\mathcal{G}$ one replaces every vertex labelled by $\mathbb{Z}$
with a vertex labelled by $\{1\}$ and a loop based at it, then the new graph of groups $\mathcal{G}_0$ still has $\pi_1(M)$ as fundamental group, and is indeed terminal. We call $\mathcal{G}_0$ the terminal graph of groups
corresponding to the prime decomposition of $M$. By Remark~\ref{big:1ended} one-ended vertex groups of $\mathcal{G}_0$  are precisely the fundamental groups of the big prime summands of $M$. 

We are now ready to prove Theorem~\ref{prime:thm}. 
Let $M=M_1\#\ldots\#M_k$ and $M'=M'_1\#\ldots \#M'_{k'}$ be the prime decompositions of $M,M'$ respectively, and let
$\mathcal{G}_0,\mathcal{G}'_0$ be the corresponding terminal graphs of groups. As usual, we set $\G=\pi_1(M)$, $\G_i=\pi_1(M_i)$,
$\G'=\pi_1(M')$, $\G_i'=\pi_1(M'_i)$. 
We first show that the conditions on $M,M'$ described in points (1), (2), (3), (4) are sufficient to ensure that $\G$ is quasi-isometric to $\G'$. 
This is obvious if $M,M'$ are both prime with finite fundamental groups, or if $M,M'$ are irreducible with infinite quasi-isometric fundamental groups, or if $M,M'\in\{S^2\times S^1,\mathbb{P}^3(\mathbb{R})\#\mathbb{P}^3(\mathbb{R})\}$, so we may suppose that both $M$ and $M'$ are not
prime and distinct from $\mathbb{P}^3(\mathbb{R})\#\mathbb{P}^3(\mathbb{R})$. In this case, Lemma~\ref{ends:lemma} ensures
that $e(\G)=e(\G')=\infty$. Moreover, 
big summands in the decompositions of $M,M'$ are exactly the one-ended vertex groups respectively
of $\mathcal{G}_0,\mathcal{G}'_0$, so the groups
$\G$ and $\G'$ are quasi-isometric by Theorem~\ref{PW:thm}.

Let us now suppose that $\G$ is quasi-isometric to $\G'$. 
Of course we have $e(\G)=e(\G')$. 
By Lemma~\ref{ends:lemma} if $e(\G)=e(\G')<\infty$ we are done, so we may suppose that both $M$ and $M'$ are not prime and distinct from $\mathbb{P}^3(\mathbb{R})\#\mathbb{P}^3(\mathbb{R})$. 
The fact that the set of quasi-isometry types of fundamental groups of big summands of $M$ coincides with 
the set of quasi-isometry types of fundamental groups of big summands of $M'$ is now a consequence of Theorem~\ref{PW:thm} and Remark~\ref{big:1ended}. This concludes the proof of Theorem~\ref{prime:thm}.

The proof of Theorem~\ref{prime2:thm} is very similar. Let $M$ be a non-prime manifold which is distinct from $\mathbb{P}^3(\mathbb{R})\#\mathbb{P}^3(\mathbb{R})$. By Lemma~\ref{ends:lemma} we have $e(\pi_1(M))=\infty$,
and we know that one-ended vertex groups of the terminal decomposition of $\pi_1(M)$ induced by the prime decomposition of $M$ correspond to fundamental groups of big summands of $M$. Therefore,
Theorem~\ref{PW:thm} implies that a group $\G$ is quasi-isometric to $\pi_1(M)$ if and only if $e(\G)=\infty$ and the set of quasi-isometry classes of one-ended vertex groups in a terminal decomposition of $\G$
is equal to the set of quasi-isometry classes of fundamental groups of big summands of $M$. This finishes the proof of Theorem~\ref{prime2:thm}. 

% ----------------------------------------------------------------

\end{document}